\documentclass[oneside,reqno]{amsart}
\usepackage{preamble}

\zcsetup{sort=false, rangesep = {~through\nobreakspace}}
\title{Hyperfiniteness of bounded-to-one actions of commutative monoids}
\author{Forte Shinko, Felix Weilacher, Jing Yu}

\begin{document}

\begin{abstract}
  A theorem of Dougherty--Jackson--Kechris states that any equivalence relation generated by a single Borel function is hypersmooth. 
  A well-known open problem is whether this can be generalized to equivalence relations generated by countable families of pairwise commuting Borel functions. 
  We give an affirmative answer in the case where the functions are bounded-to-one.
  This generalizes the theorem of Gao--Jackson on Borel actions of countable abelian groups. 
\end{abstract}

\maketitle

\tableofcontents

\section{Introduction}

The theory of Borel equivalence relations provides a framework for
studying the complexity of classification problems arising throughout mathematics.
Central to this study are the \emphd{countable} Borel equivalence relations (CBERs),
those whose classes are all countable. 

Among the most important properties in the study of Borel equivalence relations is \emphd{hypersmoothness},
which for CBERs is equivalent to \emphd{hyperfiniteness}.
These say that the relation in question can be written
as a countable increasing union of smooth and finite-class Borel equivalence relations,
respectively. 
Understanding when naturally occurring countable Borel equivalence relations are hyperfinite
has been one of the main themes of the subject.

A fundamental source of examples is provided by group actions.
If \(G\curvearrowright X\) is a Borel action of a countable group on a standard Borel space,
then the associated orbit equivalence relation
\[
  x \mathrel{E^X_G} y
  \iff
  \exists g\in G,\; gx=y
\]
is a countable Borel equivalence relation. 
In fact,
by a theorem of Feldman and Moore,
every CBER arises this way \cite{feldman.moore}.
Motivated by the Ornstein--Weiss theorem in ergodic theory \cite{ornstein_weiss},
Weiss asked the following. 
\begin{qn}[Weiss {\cite{weiss1984}}]
  Is every orbit equivalence relation of a countable amenable group hyperfinite?
\end{qn}

Weiss showed that this holds for actions of $\Z^d$, $d \in \N$. 
Following his work,
the question has been answered positively for several large classes of amenable groups:

\begin{thm}\label{thm:weiss_progress}
  Let $G$ be a countable group and $G \car X$ a Borel action on a standard Borel space.
  Then $E_G^X$ is hyperfinite provided one of the following:
  \begin{enumerate}[label=(\arabic*)]
    \item
      \textnormal{(Jackson--Kechris--Louveau \cite{jkl})}
      $G$ has polynomial growth.
    \item
      \label{wpgj}
      \textnormal{(Gao--Jackson \cite{GJ15})}
      $G$ is abelian.
    \item
      \label{polycyclic}
      \textnormal{(\jacksonfive\ \cite{CJMST23})}
      $G$ is polycyclic.
  \end{enumerate}
\end{thm}
There are more general statements known;
for instance,
the first two were generalized by Seward and Schneider \cite{SS24}
to virtually locally nilpotent groups,
which was then further generalized to locally virtually nilpotent groups
by \cite{CJMST23}.

Outside the context of group actions,
a key example of a hypersmooth Borel equivalence relation is provided by \emphd{tail equivalence}:
\begin{thm}[Dougherty--Jackson--Kechris {\cite{DJK94}}]\label{thm:tail_equiv}
  For every Borel $T : X \to X$ on a standard Borel space $X$,
  the \emphd{tail equivalence relation} of $T$ is hypersmooth,
  where $x$ and $y$ are \emphd{tail equivalent}
  if there exist $m, n \in \N$ such that $T^m(x) = T^n(y)$.
\end{thm}

In particular,
if $T$ is countable-to-one,
then its tail equivalence relation is hyperfinite.
In recent years,
this theorem has inspired a whole line of research
on the hyperfiniteness of various boundary actions,
such as the action of a Gromov hyperbolic group on its Gromov boundary
\cite{HSS20, MS20, Oya24, NV25}.

Tail equivalence relations can be naturally thought of as
``orbit equivalence relations'' of actions of the monoid $\N$.
Generally,
given a monoid action $M \car X$,
its orbit equivalence relation $E_M^X$
is the smallest equivalence relation on $X$
such that for all $m \in M$ and $x \in X$,
we have $x \mathrel{E_M^X} mx$.
When $M$ is commutative,
we have
\[
  x E_M^X y
  \iff \exists m, n \in M, \;  mx = ny.
\]
Viewed through this lens,
\zcref{thm:tail_equiv} says that
orbit equivalence relations of Borel $\N$-actions are hypersmooth,
a direct generalization of Weiss's result for $\Z$-actions.

A well-known folklore open question asks for an analogous extension
of Gao and Jackson's result (\zcref{thm:weiss_progress}\zcref{wpgj}).

\begin{qn}\label{q:main}
  Let $M \car X$ be a Borel action
  of a countable commutative monoid $M$
  on a standard Borel space $X$. 
  Is $E_M^X$ hypersmooth?
\end{qn}

Our main result is partial progress on this question.
We say that a function $T : X \to X$ is 
\emphd{at most $k$-to-one} if each fiber has size at most $k$.
We call $T$ \emphd{bounded-to-one} if it is at most $k$-to-one for some $k \in \N$. 
We call a monoid action $M \car X$ \emphd{bounded-to-one} if the map corresponding to each $m \in M$ is bounded-to-one. We emphasize that the bounds on the sizes of the fibers are allowed to vary with $m$. Observe that if $S$ is some set of maps $X \to X$, the action $\langle S \rangle \car X$ is bounded-to-one if and only if each $T \in S$ is bounded-to-one. 

\begin{thm}\label{thm:main}
Let \(M\) be a countable commutative monoid and let \(M\curvearrowright X\) be a
bounded-to-one Borel action on a standard Borel space $X$. Then \(E^X_M\) is
hyperfinite.
\end{thm}

For example, if $S, T : X \to X$ are two bounded-to-one Borel maps on a standard Borel space $X$ with $S \circ T = T \circ S$, the equivalence relation generated by $S$ and $T$ is hyperfinite. 
Even this was previously open.
Note also that \zcref{thm:main} generalizes
\zcref{wpgj} of \zcref{thm:weiss_progress} 
since in a group action each element acts as a one-to-one map. 

Recently, the most important method for proving hyperfiniteness has been
Borel asymptotic dimension,
introduced by Conley, Jackson, Marks, Seward, and Tucker-Drob \cite{CJMST23}
as a Borel analogue of Gromov's asymptotic dimension \cite{Gro93}.
Their main result states that
if a locally finite Borel graph has finite Borel asymptotic dimension,
then its connectedness relation is hyperfinite. 
See \zcref{defn:borel_asdim}
and \zcref{thm:hyperfinite_of_dimension}.
They then prove \zcref{thm:weiss_progress} \zcref{polycyclic}
by showing that Schreier graphs of free Borel actions of polycyclic groups
have finite Borel asymptotic dimension. 

We do the same in the special case of free actions
of finitely generated commutative monoids,
where we call a monoid action $M \car X$ \emphd{free}
if for all $x \in X$ and $m, n \in M$, $mx = nx$ implies $m = n$. 

\begin{thm}\label{thm:free-fg-asdim-intro}
  Let \(M\) be a finitely generated commutative monoid
  and let \(M\curvearrowright X\)
  be a free bounded-to-one Borel action on a standard Borel space.
  Then
  \[
    \asdim_B(M\curvearrowright X)<\infty.
  \]
  Equivalently, for some, hence every, finite generating set of \(M\),
  the corresponding Schreier graph has finite Borel asymptotic dimension.
\end{thm}

This result is proved below as \zcref{thm:borel_asdim_fg}. 
In that proof it is noted that one can in fact show $\asdim_B(M \car X) = \rk(M)$, the latter denoting the rank of $M$.

This result makes heavy use of the theory of local algorithms and the connection between these and Borel/continuous combinatorics developed by Bernshteyn in \cite{Ber23b, Ber23a}.
The main idea behind this connection is that one can often solve problems in Borel combinatorics by first finding a Borel proper coloring of some large power of one's graph, 
and then applying a ``local rule'' to transform this coloring into a desired combinatorial object. 
The bounded-to-one assumptions in our theorems imply that these Borel proper colorings can always be found \cite{KST99}. 
We use this connection to show the following. 
Roughly speaking, a locally-checkable labeling problem on a monoid $M$ consists of a finite set $\Sigma$ of labels and a finite set of local constraints for functions $M \to \Sigma$, and $\LOCAL_M$ consists of the set of such problems on $M$ which can be solved by the ``local rules'' described above. 
See \zcref{sec:local_algorithms} for precise definitions.
We call a function between topological spaces \emphd{clopen-preserving} if the image of every clopen set is clopen, and a monoid action \emphd{clopen-preserving} if every monoid element acts as a clopen-preserving map.

\begin{thm}\label{thm:intro_monoid_vs_group}
    Let $\Gamma$ be an abelian group and $M \subseteq \Gamma$ a finitely generated submonoid with $\Gamma = \langle M \rangle$. 
    Let $\Pi$ be a locally-checkable labeling problem on $M$.
    The following are equivalent. 
    \begin{enumerate}
        \item\label{cts_gamma} Every free continuous action $\Gamma \car X$ on a zero-dimensional Polish space $X$ admits a continuous $\Pi$-labeling.
        \item\label{cts_M} Every free continuous clopen-preserving bounded-to-one action $M \car X$ on a zero-dimensional Polish space $X$ admits a continuous $\Pi$-labeling.
        \item\label{local_gamma} $\Pi \in \LOCAL_\Gamma$.
        \item\label{local_M} $\Pi \in \LOCAL_M$.
    \end{enumerate}
    If any/all of the above hold, then every free Borel bounded-to-one action $M \car X$ on a standard Borel space $X$ admits a Borel $\Pi$-labeling.
\end{thm}

The equivalence of (\zcref{cts_gamma}) and (\zcref{local_gamma}) holds for any group $\Gamma$ and is a result of Bernshteyn \cite[Theorem 1.15]{Ber23a}.
The last part follows from (\zcref{cts_M}) by a standard change of topology argument. 
We emphasize that, though the equivalence of (\zcref{cts_gamma}) and (\zcref{cts_M}) does not mention local algorithms, our proof makes essential use of them by going through (\zcref{local_gamma}) and (\zcref{local_M}).

Our interest in \zcref{thm:intro_monoid_vs_group} is that it yields an immediate proof of \zcref{thm:free-fg-asdim-intro} thanks to the fact that free continuous actions of abelian groups are already known to have finite continuous asymptotic dimension \cite[Theorem 10.7]{CJMST23}. However, it seems likely to be useful in understanding the Borel and continuous combinatorics of commutative monoid actions more generally, especially since we already have a fairly detailed understanding of the continuous combinatorics of abelian group actions. (See e.g. \cite{gjks_cts}.) 

For example, \zcref{thm:intro_monoid_vs_group} immediately implies that free continuous clopen-preserving bounded-to-one actions $\N^d \car X$ with $X$ zero-dimensional Polish have continuous proper 4-colorings of their Schreier graphs (with respect to the standard generating set) since this is known for $\Z^d$-actions thanks to Gao--Jackson--Krohne--Seward \cite{gjks_cts}.
Relatedly, in \zcref{cor:borel_3_col} we will show by different arguments that such graphs have Borel proper 3-colorings. 

We also return briefly to the connection with boundary actions. The use of tail equivalence in the hyperbolic boundary setting is related to the fact that hyperbolic groups are ``rank 1'' and boundary points correspond to one-dimensional geodesic rays. One of the additional motivations for this paper is the fact that the boundary actions in the higher-rank case similarly correspond to the action of finitely generated commutative monoids such as $\N^d$. Thus, \zcref{thm:free-fg-asdim-intro} is a starting point to showing hyperfiniteness for such actions.

Let us now explain a bit about what is needed to pass from
\zcref{thm:free-fg-asdim-intro} to \zcref{thm:main}.
There are two additional difficulties:
\begin{enumerate}[label=(\arabic*)]
  \item
    \label{item:non_free}
    The action need not be free.
  \item
    \label{item:non_fg}
    The monoid need not be finitely generated.
\end{enumerate}
Both difficulties are classical in the theory surrounding Weiss's question,
and both are often quite hard to overcome.

Issue \zcref{item:non_free} is not problematic for actions of abelian groups
since stabilizers are constant on orbits. 
We do not have this luxury for actions of commutative monoids. 
However, it turns out that we can still reduce to the free case thanks to a theorem of R\'edei stating that finitely generated commutative monoids are Noetherian \cite{Red65}
(\zcref{thm:Redei}).
This lets us find a complete section of the action on which quotients of our monoid act freely.
See \zcref{sec:monoid} and,
e.g.,
\zcref{cor:free_action_on_stable} for details. 
This reduction (to be precise, the implication from \zcref{thm:borel_asdim_fg} to \zcref{thm:main_fg}) appeared independently as the main result in a recent arXiv preprint by Wang \cite{wang2026borel} with essentially the same proof. 


On its own, issue \zcref{item:non_fg} also does not hurt us. 
This is an instance of the notorious ``union problem'' \cite[Problem 8.24]{cber};
in general, it is open whether a countable increasing union of hyperfinite CBERs is itself hyperfinite. 
This was exactly the problem faced by Gao and Jackson
in extending Weiss's result from finitely generated abelian groups
to arbitrary countable abelian groups. 
{\jacksonfive} gave a new proof of Gao and Jackson's result by showing that Borel asymptotic dimension can be used to handle instances of the union problem. They showed that, in general, if a Borel graph is a countable increasing union of Borel subgraphs with finite Borel asymptotic dimension, that graph has hyperfinite connectedness relation (\zcref{thm:hyperfinite_of_dimension_union}). 
Thus, the special case of \zcref{thm:main}
where the action is free follows immediately from
\zcref{thm:free-fg-asdim-intro}. 

Unfortunately,
when issues \zcref{item:non_free}
and \zcref{item:non_fg} occur simultaneously,
a lot of extra work is needed.
We were not able to remove the freeness assumption in
\zcref{thm:free-fg-asdim-intro},
i.e. to show that all (not necessarily free) actions
have finite Borel asymptotic dimension. 
The difficulty here is entirely classical;
we prove in \zcref{cor:comm_monoid_asdim_eq}
that bounded-to-one Borel actions of commutative monoids
have equal classical and Borel asymptotic dimensions.
That is,
the following is open
(this is restated later as \zcref{qn:classical_asdim}):
\begin{qn}\label{qn:intro_classical_dimension}
  Let \(M\) be a finitely generated commutative monoid
  and let \(M\curvearrowright X\) be a (bounded-to-one) \(M\)-set.
  Must we have
  \(
    \asdim(M\curvearrowright X) < \infty?
  \)
\end{qn}

In the other direction,
non-finitely-generated commutative monoids are not Noetherian \cite{Budach1964} 
(consider for instance $\Z^{\oplus \omega}$),
so we cannot just copy the trick described above for the finitely generated case.

Instead, we prove the general case of \zcref{thm:main}
by proving a quantitative weakening of
\zcref{qn:intro_classical_dimension} which,
by work of Greb\'ik, Marks, Rozho\v{n}, and the second author
is still enough to handle the union problem \cite{GMRS26}.
The details are in \zcref{sec:quant_asdim, sec:local_stability, sec:proof-main}. 

We end the introduction with a summary of the contents of each section.
\zcref{sec:monoid} contains the monoid-theoretic preliminaries,
including R\'edei's theorem and an explanation of how we use it.
\zcref{sec:local_algorithms} develops a framework for local algorithms
in the setting of monoid actions.
\zcref{sec:asdim} introduces (Borel) asymptotic dimension,
the results of \cite{CJMST23},
and explains how to combine these
and the work from \zcref{sec:monoid, sec:local_algorithms}
to prove \zcref{thm:free-fg-asdim-intro}
and the finitely generated case of \zcref{thm:main}.
We emphasize that the reader only interested in these partial results
need not go beyond \zcref{sec:asdim}. 

\zcref{sec:ssi} is slightly orthogonal to the remainder of the paper.
It contains some consequences of \zcref{thm:free-fg-asdim-intro}
in Borel combinatorics centered around a generalization of the parameter
``asymptotic separation index'' introduced in \cite{CJMST23}. 
This section is not needed for the proof of \zcref{thm:main}. 

Finally, \zcref{sec:quant_asdim, sec:local_stability}
contain quantitative sharpenings of some results from
\zcref{sec:asdim, sec:monoid} respectively. 
These are combined in \zcref{sec:proof-main} to prove \zcref{thm:main}. 

\section{Monoid generalities}\label{sec:monoid}
Let $\sim$ be an equivalence relation on a monoid $M$.
We say that $\sim$ is a \emphd{right congruence}
(resp. \emphd{left congruence})
if for all $m, m', a \in M$,
if $m \sim m'$,
then $ma \sim m'a$
(resp. $am \sim am'$).
We say that $\sim$ is a \emphd{congruence}
if it is both a left congruence and right congruence.
Given a monoid $M$ and a congruence $\sim$ on $M$,
the monoid operation on $M$ descends to the quotient $M/{\sim}$.

We will make use of the following theorem of R\'edei,
where a monoid is \emphd{Noetherian}
if every infinite ascending chain of congruences on $M$ is eventually constant,
or equivalently if every congruence on $M$ is finitely generated.
\begin{thm}[R\'edei's theorem \cite{Red65}]\label{thm:Redei}
  Every finitely generated commutative monoid is Noetherian.
\end{thm}

A monoid is \emphd{right-cancellative}
(resp. \emphd{left-cancellative})
if for all $m, m', a \in M$,
if $ma = m'a$ 
(resp. $am = am'$),
then $m = m'$.
A monoid is \emphd{cancellative}
if it is both left-cancellative and right-cancellative.
The forgetful functor from groups to monoids has a left adjoint,
sending every monoid $M$ to its \emphd{group completion} $M^\gp$;
explicitly,
it is the monoid with presentation
$\ev{M, M^{-1} : mm^{-1} = m^{-1}m = 1}$.
A commutative monoid is cancellative if and only if it is a submonoid of some group, in which case $M^\gp$ is the minimal such group. 
For a finitely generated commutative monoid, we write
\[
  \rk(M) := \rank_{\Z}(M^{\gp})
\]
for the torsion-free rank of its group completion.

For the most part, we have tried to state each lemma in this paper for as general a class of monoids as possible. 
However, when it comes to proving the main theorem, we will only ever need to consider cancellative commutative monoids (Equivalently, submonoids of abelian groups). 
The reader may find it useful to focus on these examples for the sake of concreteness. 

Given a subset $S$ of a monoid $M$,
we say that $\hat s \in S$ is an \emphd{lcm of $S$}
if for every $s \in S$,
there is some $s' \in S$
such that $ss' = \hat s$.
Note that every finite subset $S$ of a commutative monoid $M$
is contained in a finite subset with an lcm,
namely $\{\prod S' : S' \subseteq S\}$ with lcm $\prod S$.

An \emphd{$M$-set} is a set equipped with an $M$-action.
Recall from the introduction that
the \emphd{orbit equivalence relation} of an $M$-set $X$,
denoted $E_M^X$,
is the smallest equivalence relation on $X$
such that for all $x \in X$ and all $m \in M$,
we have $x \mr E_M^X mx$.
The equivalence classes of the orbit equivalence relation
are called \emphd{orbits}.
If $M$ is commutative,
then $x$ and $x'$ are in the same orbit
iff there exist $m, m' \in M$ such that $mx = m'x'$.
An $M$-set is \emphd{transitive} if it has exactly one orbit.
A set $Y\subseteq X$ is a \emphd{complete section} for $E_M^X$
if it meets every $E_M^X$-class.
For $x,y \in X$, we write $x \leq y$ if $y \in Mx$. 
A set $Y \subseteq X$ is called \emphd{$M$-invariant} if $MY \subseteq Y$. 

\begin{defn}\label{defn:sim_x}
  Let $M$ be a monoid,
  $X$ an $M$-set,
  and $x \in X$.
  The equivalence relation $\sim_x$ on $M$
  is defined by $m \sim_x m' \iff mx = m'x$.
\end{defn}

\begin{lem}\label{lem:congruence}
  Let $M$ be a monoid,
  let $X$ be an $M$-set,
  and let $x \in X$.
  Then $\sim_x$ is a left congruence.
  In particular,
  if $M$ is commutative 
  then $\sim_x$ is a congruence.
\end{lem}

\begin{proof}
    Suppose $m\sim_x m'$, i.e., $mx=m'x$. For any $a\in M$, $(am)x=a(mx)=a(m'x)=(am')x$, so $am\sim_x am'$. Thus $\sim_x$ is a left congruence.
\end{proof}

\begin{lem}\label{lem:sim_mono}
  Let $M$ be a commutative monoid,
  $X$ an $M$-set,
  $x \in X$,
  and $m \in M$.
  Then ${\sim}_x \subseteq {\sim}_{mx}$.
\end{lem}

\begin{proof}
  If $nx = n'x$,
  then $nmx = mnx = mn'x = n'mx$. 
\end{proof}

\begin{prop}
  Let $M$ be a finitely generated commutative monoid,
  and let $X$ be a transitive $M$-set.
  Then there is some $x \in X$ such that for all $x' \in X$,
  we have ${\sim}_{x'} \subseteq {\sim}_x$.
\end{prop}
\begin{proof}
  By \zcref{lem:congruence},
  since $M$ is commutative,
  each $\sim_x$ is a congruence on $M$. 
  Since $M$ is Noetherian by \zcref{thm:Redei}, 
  the poset of congruences $\{{\sim}_x : x \in X\}$ under inclusion
  has a maximal element.
  We claim this poset is directed,
  so the maximal element is a maximum.
  Indeed, given $x,x' \in X$, by transitivity there are $m,m' \in M$
  such that $mx = m'x' =: y$,
  and then $\sim_x,\sim_{x'} \subseteq \sim_y$ by \zcref{lem:sim_mono}. 
\end{proof}

\begin{prop}
  Let $M$ be a finitely generated commutative monoid,
  and let $X$ be a transitive $M$-set.
  The set $Y := \{x \in X \mid \textnormal{${\sim}_x$ is maximum}\}$
  is $M$-invariant.
\end{prop}

\begin{proof}
    Immediate from \zcref{lem:sim_mono}. 
\end{proof}

\begin{defn}
  Let $M$ be a finitely generated commutative monoid.
  Let $X$ be an $M$-set. Call a point $x \in X$ \emphd{stable}
  if $\sim_x$ is maximum among $\{{\sim}_{y} \mid y \mathrel{E}^X_M x\}$. Let $\cS(X) \subseteq X$ denote the set of stable points.  
\end{defn}

By the previous two propositions,
$\cS(X)$ is an $M$-invariant complete section for $E_M^X$. 
The following refinement of this will be used in \zcref{sec:ssi},
but it is not needed for the proof of \zcref{thm:main}. 
For $A$ a set, we write $B \Subset A$ to denote that $B$ is a finite subset of $A$. 

\begin{prop}\label{prop:stable_born}
  Let $M$ be a finitely generated commutative monoid,
  let $X$ be an $M$-set,
  let $Y \subseteq X$ be an $M$-invariant complete section of $E_M^X$,
  and let $R \Subset M$ be a finite generating set with lcm $\hat r$.
  \begin{enumerate}[label=(\arabic*)]
    \item
      For every $x \in X$,
      there is some $k \in \N$ such that $\hat r^kx \in Y$.
    \item
      For $x \in X$,
      let $t(x) \in \N$ be the least such $k$. 
      If $x \in X$ and $x' \in Rx$,
      then $|t(x) - t(x')| \le 1$.
  \end{enumerate}
\end{prop}
\begin{proof}
  \leavevmode
  \begin{enumerate}[label=(\arabic*)]
    \item
      Since $Y$ is a complete section,
      let $y \in Y$ be in the same orbit as $x$.
      There are $m,m' \in M$ such that $mx = m'y$.
      Write $m = \prod_{i < k} r_i$ for $k \in \N$ and $r_i \in R$.
      For each $i$, let $r_i' \in M$ with $r_i r_i' = \hat r$.
      Then
      \[
        \hat r^k x
        = \qty(\prod_{i < k} r_i' \cdot  \prod_{i < k} r_i)x
        \in M(mx)
        = M(m'y)
        \subseteq MY
        \subseteq Y.
      \]
    \item
      Since
      $\hat r^{t(x)} x' \in \hat r^{t(x)} R x \subseteq RY \subseteq Y$,
      it follows that $t(x') \le t(x)$.
      On the other hand,
      let $r,r' \in R$ with $x' = rx$ and $rr' = \hat r$.
      Then $\hat r^{t(x')+1} x = r'\hat r^{t(x')} x' \in r'Y \subseteq Y$,
      so $t(x) \le t(x') + 1$.
\end{enumerate}
\end{proof}

\begin{defn}\label{defn:go_to_stable}
  Let $M$ be a finitely generated commutative monoid,
  $X$ an $M$-set,
  and $R \Subset M$ a finite generating set with lcm $\hat r$.
  Let $t: X \to \N$ be defined as in \zcref{prop:stable_born} for $Y = \cS(X)$. 
  Let $T : X \to \cS(X)$ be defined by $T(x) = \hat r^{t(x)} x$. 
\end{defn}

The useful thing about $\cS(X)$ is that
it gives us free actions of certain quotients of our commutative monoid.
Recall the following from the introduction.

\begin{defn}
  Let $M$ be a monoid and $X$ an $M$-set.
  A point $x \in X$ is called \emphd{free}
  if $\sim_x$ is the identity relation on $M$.
  That is, for all $m,m' \in M$, $mx = m'x$ if and only if $m = m'$. 
  Let $\cF(X) := \{x \in X \mid \textnormal{$x$ is free}\}$.
  We call $X$ (or the action) \emphd{free} if $X = \cF(X)$. 
\end{defn}

Note that $M$ is right-cancellative iff there is a nonempty free $M$-set:
on one hand,
if $M$ is right-cancellative,
then $M$ itself is a free $M$-set,
and on the other hand,
if $X$ is a free $M$-set and $x \in X$, 
then for every $m, m', r \in M$,
if $mr = m'r$,
then $mrx = m'rx$,
so $m = m'$ by freeness.

\begin{lem}\label{lem:action_descend}
  Let $M$ be a commutative monoid and $X$ an $M$-set.
  Let $x \in X$ and $\sim$ a congruence on $M$. 
  \begin{enumerate}[label=(\arabic*)]
    \item
      If $\sim \subseteq \sim_x$,
      then the action $M \car Mx$ descends to an action of $M/{\sim}$.
    \item
      \label{item:lem:action_descend_free}
      If $\sim = \sim_x$,
      then $x$ is free with respect to the action $M/{\sim} \car Mx$. 
  \end{enumerate}
\end{lem}

\begin{proof}
  \leavevmode
  \begin{enumerate}[label=(\arabic*)]
    \item
      Let $y \in Mx$.
      By \zcref{lem:sim_mono},
      ${\sim} \subseteq {\sim}_y$.
      We need to show that if $m \sim m'$ then $my = m'y$,
      which is exactly what the previous inclusion says.
    \item
      We need to show that if $mx = m'x$ then $m = m'$ in $M/\sim$.
      This is exactly the inclusion $\sim_x \subseteq \sim$. 
  \end{enumerate}
\end{proof}

\begin{cor}\label{cor:free_action_on_stable}
  Let $M$ be a finitely generated commutative monoid and $X$ a transitive $M$-set.
  Let $\sim = \sim_x$ for some,
  equivalently any,
  $x \in \cS(X)$. 
  The action $M \car \cS(X)$ descends to a free action of $M/\sim$. 
\end{cor}
\begin{proof}
  By hypothesis,
  assumption \zcref{item:lem:action_descend_free} of \zcref{lem:action_descend}
  holds for every $x \in \cS(X)$. 
\end{proof}

\section{Local algorithms on monoids}
\label{sec:local_algorithms}

Linial's deterministic $\LOCAL$ model \cite{linial} is a model of distributed computing. In it, one pictures a graph as a network of computers where each can communicate with its neighbors, and measures the number of rounds of communication needed to solve some combinatorial problem on the graph. 
Bernshteyn showed that there are deep connections between the $\LOCAL$ model and descriptive combinatorics \cite{Ber23b, Ber23a}. 

In this section we formalize a version of the deterministic $\LOCAL$ model tailored to monoids and monoid actions and develop the basic theory of it, including its connection to the descriptive combinatorics of monoid actions in the spirit of Bernshteyn's work. 
The fundamental units of this theory are the following procedures for transforming data locally. 

\begin{defn}
  Let $M$ be a monoid and let $R \Subset M$.
  An \emphd{$R$-local algorithm} is a function
  \[A : \Sigma^R \to \Lambda,\]
  where $\Sigma$ and $\Lambda$ are finite sets.
  The set $\Sigma$,
  also denoted $\Sigma_A$,
  is the \emphd{set of input labels},
  and the set $\Lambda$,
  also denoted $\Lambda_A$,
  is the \emphd{set of output labels}.
  The set $R$,
  also denoted $R_A$,
  is the \emphd{locality} of the local algorithm.

  Given $S \subseteq M$,
  an \emphd{$S$-local algorithm}
  is an $R$-local algorithm for some $R \Subset S$.
\end{defn}

The final sentence of this definition may appear to introduce some ambiguity in the case $S \Subset M$, but see \zcref{rmk:bigger_locality}. 

We can compose local algorithms:

\begin{defn}\label{def:compose}
  Let $M$ be a monoid,
  and let $A : \Sigma^R \to \Lambda$
  and $B : \Lambda^S \to \Delta$
  be local algorithms.
  We define the $RS$-local algorithm
  $B*A : \Sigma^{RS} \to \Delta$
  by \[(B*A)(l) = B(s \mapsto A(r \mapsto l(rs))).\]
\end{defn}

We can run a local algorithm at a point in a labeled $M$-set.
We do not require that the $M$-set be free.
\begin{defn}
  Let $M$ be a monoid,
  and let $X$ be an $M$-set.
  Given an $M$-local algorithm $A$
  and a function $l : X \to \Sigma_A$,
  define the function
  \(
    A * l : X \to \Lambda
  \) by 
  \[
    x \mapsto A(r \mapsto l(rx))
  \]
\end{defn}

\begin{rmk}\label{rmk:bigger_locality}
  Let $M$ be a monoid and $R \subseteq R' \Subset M$. 
  Let $A : \Sigma^{R} \to \Lambda$ be an $R$-local algorithm. 
  Let $A' : \Sigma^{R'} \to \Lambda$ be the $R'$-local algorithm defined by $A'(l) = A(l \upharpoonright R)$. 
  Then for any $M$-set $X$ and labeling
  $l: X \to \Sigma$, $A' * l = A * l$.
  In this way we will sometimes implicitly consider an $R$-local algorithm to also be an $R'$-local algorithm when $R \subseteq R' \Subset M$. 
\end{rmk}

\begin{defn}
  Let $M$ be a monoid.
  An \emphd{$M$-LCL} is an $M$-local algorithm
  whose set of output labels is $2$.
  We often identify an LCL $\Pi$ with the set $\Pi^{-1}(\{1\})$.
\end{defn}

LCL stands for \emphd{locally checkable labeling problem}. 
We think of an LCL as a set of local constraints which we might like some labeling to satisfy.  
By running an LCL on a labeled $M$-set,
we can check whether it does in fact satisfy these constraints. 

\begin{defn}
  Let $M$ be a monoid,
  let $X$ be an $M$-set,
  let $\Pi$ be an $M$-LCL,
  and let $l : X \to\Sigma_\Pi$.
  Given $x \in X$,
  we say that $l$ is a \emphd{$\Pi$-labeling at $x$}
  if
  \[
    (\Pi * l)(x) = 1.
  \]
  For $Y \subseteq X$, 
  we say that $l$ is a \emphd{$\Pi$-labeling of $Y$}
  if it is a $\Pi$-labeling at every $x \in Y$.
\end{defn}

One particular LCL of interest involves injective labeling:
\begin{defn}
  Let $M$ be a monoid,
  let $R \Subset M$,
  and let $n \in \N$.
  The LCL $\Pi_{R \hra n} : n^R \to 2$
  is defined by
  \[
    \Pi(l) = 1
    \iff
    \text{$l$ is injective}.
  \]
\end{defn}

The following is a natural notion of reduction between LCLs in this setting. 

\begin{defn}
  Let $M$ be a monoid,
  let $\Pi$ and $\Pi'$ be $M$-LCLs,
  and let $R \Subset M$.
  An \emphd{$R$-local reduction from $\Pi'$ to $\Pi$}
  is an $R$-local algorithm
  \[
    A : (\Sigma_\Pi)^R \to \Sigma_{\Pi'}
  \]
  such that for every $l: M \to \Sigma_\Pi$
  which is a $\Pi$-labeling of $R$,
  we have that $A*l$ is a $\Pi'$-labeling at $1$.
  
  That is, if $\Pi * l = 1$ on $R$, then
  \[
    (\Pi' * A * l)(1) = 1.
  \]
\end{defn}

A reduction from $\Pi'$ to $\Pi$ allows us to locally transform $\Pi$-labelings into $\Pi'$-labelings as recorded by the next lemma. 

\begin{lem}\label{lem:apply_reduction}
  Let $M$ be a monoid,
  let $\Pi$ and $\Pi'$ be $M$-LCLs,
  and let $R \Subset M$.
  Let $A$ be an $R$-local reduction from $\Pi'$ to $\Pi$.
  Let $X$ be an $M$-set, let $x \in X$, and let $l : X \to \Sigma_\Pi$ such that $l$ is a $\Pi$-labeling of $Rx$.
  Then $A * l$ is a $\Pi'$-labeling at $x$.
\end{lem}

\begin{proof}
  Define $l_x : M \to \Sigma_\Pi$ by $l_x(m) = l(mx)$. 
  For any $r \in R$,
  \[
    (\Pi * l_x)(r)
    = \Pi(s \mapsto l_x(sr))
    = \Pi(s \mapsto l(srx))
    = (\Pi * l)(rx)
    = 1,
  \]
  so $(\Pi' * A * l_x)(1) = 1$.
  The same calculation shows $(\Pi' * A * l)(x) = (\Pi' * A * l_x)(1)$,
  so we are done.
\end{proof}

We can compose reductions:

\begin{lem}\label{lem:compose_reduction}
    Let $M$ be a monoid.
    Let $\Pi,\Pi',\Pi''$ be $M$-LCLs. 
    Let $R,R' \Subset M$. 
    Suppose $A$ is an $R$-local reduction from $\Pi'$ to $\Pi$ 
    and 
    $A'$ is an $R'$-local reduction from $\Pi''$ to $\Pi'$. 
    Then $A' * A$ is an $RR'$-local reduction from $\Pi''$ to $\Pi$.
\end{lem}

\begin{proof}
    Let $l : M \to \Sigma_\Pi$ be a $\Pi$-labeling of $RR'$. 
    It suffices to show that $A * l$ is a $\Pi'$-labeling of $R'$. 
    Let $r' \in R'$, so that we need to show $A * l$ is a $\Pi'$-labeling at $r'$. Since $Rr' \subseteq RR'$, we are done by \zcref{lem:apply_reduction}. 
\end{proof}

The discussion so far has been about labeled $M$-sets.
To produce a version of the $\LOCAL$ model, we need to explain what it means for a local algorithm to ``solve'' an LCL on an $M$-set which may not have an input labeling. 
The choice made in this situation by Linial was to work with graphs $G$ with vertices labeled by an arbitrary injection $V(G) \hookrightarrow |V(G)| \in \N$.
One then measures the locality of algorithms which produce $\Pi$-labelings from such labelings as a function of $|V(G)|$.

Since in a local algorithm each vertex will only ``see'' some bounded size neighborhood of itself, it is usually enough to assume only that these neighborhoods (rather than the entire graph) are injectively labeled. 
This brings us to the next definition. 

\begin{defn}
  Let $M$ be a monoid,
  let $\Pi$ be an LCL on $M$,
  and let $R \Subset M$.
  The \emphd{$R$-round complexity of $\Pi$}
  is the function
  \[
    \rounds_{\Pi, R} : \N \to \N \cup \{\infty\}
  \]
  defined as follows: 
  for $n \in \N$,
  $\rounds_{\Pi, R}(n)$ is the least $t \in \N$,
  if it exists,
  such that there is an $R^t$-local reduction
  from $\Pi$ to $\Pi_{R^t \hra n}$. 
  If no such $t$ exists, we set
  $\rounds_{\Pi,R}(n)=\infty$.

  For $f : \N \to \N \cup \{\infty\}$, 
  we say $\Pi \in \LOCAL_R(f)$ if
  \[
    \rounds_{\Pi, R}(n) \le f(n) \qquad \text{for all $n \in \N$.}
  \]
\end{defn}

\begin{lem}\label{lem:rounds_mono}
  Let $M$ be a monoid,
  and let $\Pi$ be an LCL on $M$.
  Let $R \subseteq R' \Subset M$ and $n' \le n \in \N$. 
  Then
  \[
    \rounds_{\Pi,R'}(n') \le \rounds_{\Pi,R}(n).
  \]
\end{lem}

\begin{proof}
    Let $t=\rounds_{\Pi,R}(n)$, and suppose $t<\infty$. 
    The inclusion $n' \hra n$ gives a $1$-local reduction from $\Pi_{R^t \hra n}$ to $\Pi_{(R')^t \hra n'}$.
    Therefore by \zcref{lem:compose_reduction} there is an $R^t$-local reduction from $\Pi$ to $\Pi_{(R')^t \hra n'}$, hence an $(R')^t$-local one by \zcref{rmk:bigger_locality}.
\end{proof}

\begin{lem}\label{lem:local_reduction_additive}
  Let $M$ be a monoid,
  and let $\Pi$ and $\Pi'$ be $M$-LCLs.
  Let $1 \in R \Subset M$.
  If $\Pi'$ has an $R^s$-local reduction to $\Pi$,
  then for all $n \in \N$, 
  $\rounds_{\Pi',R}(n) \le \rounds_{\Pi, R}(n) + s$. 
\end{lem}
\begin{proof}
  Let $t = \rounds_{\Pi,R}(n)$.
  By \zcref{lem:compose_reduction},
  $\Pi'$ has a $R^{t + s}$-local reduction to $\Pi_{R^t \hra n}$,
  hence to $\Pi_{R^{t+s} \hra n}$. 
\end{proof}

The following will be important
so we can shift the locality to the positive cone:

\begin{lem}\label{lem:locality_shift}
  Let $M$ be a monoid,
  let $\Pi$ be an $M$-LCL,
  let $R \Subset M$,
  and let $m \in M$ be such that $m$ commutes with every element of $R$
  and right multiplication by $m$ is injective. 
  Then
  \[
    \rounds_{\Pi, Rm} \le \rounds_{\Pi, R}.
  \]
  In particular this conclusion holds for every $m \in M$
  if $M$ is cancellative and commutative. 
\end{lem}
\begin{proof}
  Let $n \in \N$ and $t = \rounds_{\Pi,R}(n)$.
  By the injectivity assumption,
  the identity function on $n$ is a $\{m^t\}$-local reduction 
  from $\Pi_{R^t \hra n}$ to $\Pi_{R^t m^t \hra n}$,
  so by \zcref{lem:compose_reduction},
  there is a $R^t m^t$-local reduction from $\Pi$ to $\Pi_{R^t m^t \hra n}$.
  By the commutativity assumption, $R^tm^t = (Rm)^t$, so we are done. 
\end{proof}

It is also important to note that changing the ambient monoid
does not change the complexity of an LCL. 

\begin{lem}
  Let $M \subseteq N$ be monoids.
  Let $\Pi$ be an LCL on $M$,
  viewed also an LCL on $N$,
  and $R \Subset M$.
  Then $\rounds_{\Pi,R}$ is the same whether computed in $M$ or $N$. 
\end{lem}
\begin{proof}
  For every $t$, the locality $R^t$ is contained in $M$.
  Thus an $R^t$-local reduction is the same finite function
  whether the ambient monoid is $M$ or $N$.
\end{proof}

The previous two lemmas will be used in the following way:

\begin{lem}\label{lem:rounds_monoid_vs_rounds_group}
  Let $\Gamma$ be an abelian group and $M$ a submonoid. 
  Let $\Pi$ be an $M$-LCL.
  Let $R \Subset M$ have an lcm,
  and as usual define $R^{\pm1} \subseteq \Gamma$
  as the set of elements of $R$ and their inverses.
  Then
  \[
    \rounds_{\Pi,R^{\pm 1}}
    \le \rounds_{\Pi,R}
    \le 2 \rounds_{\Pi,R^{\pm 1}}.
  \]
\end{lem}

\begin{proof}
  The first inequality follows from \zcref{lem:rounds_mono},
  so we focus on the second. 
  Let $\hat r \in R$ be an lcm of $R$.
  Observe that $R^{\pm 1} \hat r \subseteq R^2$.
  Thus $\rounds_{\Pi, R^{\pm 1}} \ge \rounds_{\Pi, R^2}$
  by \zcref{lem:locality_shift, lem:rounds_mono},
  and clearly $2 \rounds_{\Pi,R^2} \ge \rounds_{\Pi, R}$,
  so we are done. 
\end{proof}

The following lemma,
whose proof is routine,
explains the connection between local algorithms and Borel/continuous combinatorics.

\begin{lem}\label{lem:local_borel}
  Let $M$ be a monoid and $R \Subset M$.
  Let $A : \Sigma^R \to \Lambda$ be an $R$-local algorithm.
  \begin{enumerate}
      \item  Let $M \car X$ be a Borel action of $M$ on a standard Borel space $X$. If $l : X \to \Sigma$ is Borel then so is $A * l$.
      \item Let $M \car X$ be a continuous action of $M$ on a topological space $X$. If $l : X \to \Sigma$ is continuous then so is $A * l$. 
  \end{enumerate}
  
\end{lem}

\begin{defn}\label{defn:R-free}
    Let $M$ be a monoid, $R \subseteq M$, and $X$ an $M$-set. A point $x \in X$ is called \emphd{$R$-free} if for all $a,b,r \in R$, $arx = brx$ implies $a = b$. 
    Denote the set of $R$-free points in $X$ by $\cF_R(X)$. 
\end{defn}

Note that if $M$ is right-cancellative it suffices to check that $ax = bx$ implies $a = b$ for all $a,b \in R^2$.
Note also that, clearly, if $X$ is free then $\cF_R(X) = X$ for any $R$.

Below we use $\log^*$ for the \emphd{iterated logarithm}, i.e., the number of times the logarithm function must be iteratively applied to $n$ before the result becomes at most $1$. 
The notation $\poly_{p_1,\ldots,p_j}(s)$ denotes a quantity bounded above by a polynomial in $s$, with degree and constants depending only on the parameters $p_1, \ldots,p_j$. 
\begin{prop}[essentially Bernshteyn {\cite[Theorem 2.10]{Ber23b}}]\label{prop:local_to_borel}
    Let $M$ be a monoid. Let $1 \in R \Subset M$ and $\Pi$ be an $M$-LCL with $\Pi \in \LOCAL_R(C \log^* n)$ for some $C > 0$.
    For any $k \in \N$,
    there exists $t \in \poly_{k,|R|}(C)$ with the following property:
    Let $M \car X$ be a Borel action of $M$ on a standard Borel space $X$
    such that the action of each $m \in R$ is at most $k$-to-one. 
    There is a Borel $f : X \to \Sigma_\Pi$
    such that $f$ is a $\Pi$-labeling of 
    $\cF_{R^t}(X)$.
\end{prop}

\begin{proof}
    For any $t \in \N$, let $G_t$ be the graph on $X$ with an edge between $x,y \in X$ whenever $x \neq y$ and there is a $z \in X$ with $x,y \in R^t z$. It is easy to see this graph is Borel. 
    
    We first bound the degree of $G_t$. 
    For any $x \in X$ and $m \in R^t$, the action of $m$ is at most $k^t$-to-one, so there are at most $|R^t|k^t$ points $z \in X$ with $x \in R^t z$. Each such $z$ contributes at most $(|R^t|-1)$ neighbors to $x$, so the degree of $x$ is at most $|R|^t(|R|^t - 1)k^t < (|R|^2 k)^t$. Let $a = |R|^2 k$.

    Let $t \in \N$ be such that $\rounds_{\Pi,R}(a^t) \le C \log^*(a^t) \le t$. 
    We claim this $t$ works.
    Let $A$ be an $R^t$-local reduction from $\Pi$ to $\Pi_{R^t \hra a^t}$.
    By \cite[Theorem 4.6]{KST99}, there is a Borel proper coloring $l : X \to a^t$ of $G_t$.
    We claim $f := A * l$ is as desired. It is Borel by \zcref{lem:local_borel}. 
    Fix $x \in \cF_{R^t}(X)$. By \zcref{lem:apply_reduction} it suffices to show that for each $m \in R^t$, $l$ is a $\Pi_{R^t \hra a^t}$-labeling at $mx$. We have that 
    $l$ is injective on $R^t mx$ since this set is a clique in $G_t$. 
    Furthermore the map $m' \mapsto m'mx$ from $R^t$ to $R^t m x$ is injective by definition of $R^t$-free.

    It remains to check that we can take $t \in \poly_a(C)$.
    We may assume $C \ge 1$ and $a \ge e$. 
    We claim then that $t = e^3 C^2 \log(a)$ works.
    First observe that since $\log^*(x) \le x$ for all $x \ge 1$,
    $\log^*(x) = \log^*(e^{e^{\log(\log(x))}}) = \log^*(\log(\log(x)) + 2 \le \log(\log(x)) + 2$ for $x \ge e^2$. 
    Now 
    \begin{align*}
        C \log^*(a^t) \le C(\log(\log(a^t)) + 2) = C(\log(t \log(a)) + 2) = C(\log(C^2 \log^2(a) e^3) + 2) \\ 
        = C(2\log(C) + 2\log(\log(a)) + 5) \le 9C^2 \log(a) < t,
    \end{align*}
    where the second to last inequality used $\log(x) \le x$ a few times.
\end{proof}

The reader may have noticed that the $\log^*$ in the assumption of this proposition is overkill; any $o(\log n)$ function would work in the same way. However, it turns out that in practice there is no difference between these hypotheses. 
Indeed, by a theorem of Chang and Pettie \cite{CP} the latter implies the former.

We will ultimately need the quantitative precision of \zcref{prop:local_to_borel} in the proof of \zcref{thm:main}.
However, if we ignore constants and restrict to free actions then softer and more elegant statements can be made, and these turn out to suffice for the case where the monoid is finitely generated. We turn to these now. 

\begin{lem}
    Let $M$ be a finitely generated monoid and $\Pi$ an $M$-LCL. Let $f : \N \to \N$. The following are equivalent.
    \begin{enumerate}
        \item For some generating set $1 \in R \Subset M$, $\Pi \in \LOCAL_R(O(f(n)))$.
        \item For any generating set $1 \in R \Subset M$, $\Pi \in \LOCAL_R(O(f(n)))$.
    \end{enumerate}
\end{lem}

\begin{proof}
    Given (1) and some other generating set $1 \in S \Subset M$, there exists $c \in \N$ such that $R \subseteq S^c$, and then $\rounds_{\Pi, S} \le c \rounds_{\Pi,S^c} \le c \rounds_{\Pi,R}$. 
\end{proof}

This allows us to make the following definition. 

\begin{defn}
    Let $M$ be a finitely generated monoid and $\Pi$ an $M$-LCL. We say $\Pi \in \LOCAL_M$ if $\Pi \in \LOCAL_R(O(\log^* n))$ for some, equivalently every, finite generating set $1 \in R \Subset M$. 
\end{defn}

\begin{prop}[Bernshteyn \cite{Ber23b}]\label{prop:local_to_borel_global}
    Let $M$ be a finitely generated monoid and $\Pi$ an $M$-LCL. If $\Pi \in \LOCAL_M$, then any free bounded-to-one Borel action of $M$ on a standard Borel space admits a Borel $\Pi$-labeling.
\end{prop}

\begin{proof}
    Immediate from \zcref{prop:local_to_borel}.
\end{proof}

In fact, we can get a topological version of this result for nice enough actions:

\begin{prop}[Bernshteyn \cite{Ber23b}]\label{prop:local_to_cts}
    Let $M$ be a finitely generated monoid and $\Pi$ an $M$-LCL. If $\Pi \in \LOCAL_M$, then any free bounded-to-one continuous clopen-preserving action of $M$ on a zero-dimensional Polish space $X$ admits a continuous $\Pi$-labeling.
\end{prop}

\begin{proof}
    We repeat the proof of \zcref{prop:local_to_borel}. The only new detail is that we need to ensure that $G_t$ has a continuous proper $a^t$-coloring. 
    By \cite[Lemma 2.3]{Ber23a}, it suffices to show that $G_t$ is a \emphd{continuous graph}, meaning if $Y \subseteq X$ is clopen then so is $N_{G^t}(Y) := \{x \in X \mid \exists y \in Y, \{x,y\} \in G_t\}$. 

    By freeness, for $x,y \in X$ we have $\{x,y\} \in G_t$ if and only if there is a $z \in X$ and $a \neq b \in R^t$ with $x = az$ and $y = bz$. 
    Therefore
    \[ N_{G^t}(Y) = \bigcup_{a \neq b \in R^t} a[b\inv[Y]]. \]
    The right hand side is clopen if $Y$ is since each $a \in R^t$ is continuous and clopen-preserving, so this suffices. 
\end{proof}

We can now prove \zcref{thm:intro_monoid_vs_group}, restated here:
\begin{thm}\label{thm:monoid_vs_group}
    Let $M$ be a finitely generated cancellative commutative monoid and $\Pi$ an $M$-LCL. The following are equivalent
    \begin{enumerate}
        \item\label{cts_gamma'} Every free continuous action $M^{\gp} \car X$ on a zero-dimensional Polish space $X$ admits a continuous $\Pi$-labeling.
        \item\label{cts_M'} Every free continuous clopen-preserving bounded-to-one action $M \car X$ on a zero-dimensional Polish space $X$ admits a continuous $\Pi$-labeling.
        \item\label{local_gamma'} $\Pi \in \LOCAL_{M^{\gp}}$.
        \item\label{local_M'} $\Pi \in \LOCAL_M$.
    \end{enumerate}
    If any/all of the above hold, then every free Borel bounded-to-one action $M \car X$ on a standard Borel space $X$ admits a Borel $\Pi$-labeling.
\end{thm}

\begin{proof}
    (\zcref{local_gamma'}) and (\zcref{local_M'}) are equivalent by \zcref{lem:rounds_monoid_vs_rounds_group}. 
    (\zcref{local_gamma'}) and (\zcref{cts_gamma'}) are equivalent by \cite[Theorem 1.15]{Ber23a}. 
    (\zcref{cts_M'}) $\implies$ (\zcref{cts_gamma'}) is clear since in a continuous group action each element acts by a homeomorphism, and
    (\zcref{local_M'}) $\implies$ (\zcref{cts_M'}) is \zcref{prop:local_to_cts}. 
    The final statement follows from \zcref{prop:local_to_borel_global}. 
\end{proof}




\section{Asymptotic dimension}\label{sec:asdim}

For $G$ a graph with vertex set $X$, let $\rho_G$ denote the path metric on $G$. We denote the ball of radius $s$ at the point $x \in X$ by $B_G(x, s)$ and the connectedness equivalence relation of $G$ by $E_G$. For $s \in \N$, $G^s$ denotes the graph on $X$ consisting of pairs $\{x,y\}$ with $\rho_G(x,y) \leq s$. 

\begin{defn}[Gromov \cite{Gro93}]
    Let $G$ be a locally finite graph on a vertex set $X$. Let $s,d \in \N$.
    \begin{itemize}
        \item An \emphd{uncolored dimension $d$ witness at scale $s$} for $G$ is a function $f:X \to X$ such that 
        \begin{enumerate}
            \item $\sup_x \rho_G(x,f(x)) < \infty$;
            \item For each $x \in X$, $|f[B_G(x,s)]| \le d + 1$. 
        \end{enumerate}
        The value of the supremum in (1) is called a \emphd{radius} of the witness. More generally any upper bound is called a radius.

        \item A \emphd{colored dimension $d$ witness at scale $s$} for $G$ is a partition $X = U_0 \sqcup \cdots \sqcup U_d$ such that there exists an $R \in \N$ such that for each $i \le d$, each connected component of $G^s \uhr U_i$ has $\rho_G$-diameter at most $R$. 
        Any such $R$ is called a \emphd{diameter} of the witness.
    \end{itemize}
\end{defn}

\begin{defn}[Conley--Jackson--Marks--Seward--Tucker-Drob \cite{CJMST23}]\label{defn:borel_dim}
    Let $G$ be a locally finite Borel graph on a standard Borel space $X$. Let $s,d \in \N$.
    \begin{itemize}
        \item A \emphd{Borel uncolored dimension $d$ witness at scale $s$} for $G$ is a witness $f : X \to X$ which is Borel as a function.
        \item A \emphd{Borel colored dimension $d$ witness at scale $s$} for $G$ is a witness $X = U_0 \sqcup \cdots \sqcup U_d$ for which each $U_i$ is Borel.
    \end{itemize}
\end{defn}

A more common definition of uncolored dimension $d$ witness at scale $s$ is an equivalence relation $E$ on $X$ such that 
    \begin{enumerate}
        \item $\sup_{C \in X/E} \text{diam}_{\rho_G}(C) < \infty$;
        \item For each $x \in X$, $B_G(x,s)$ meets at most $d+1$ many $E$-classes.
    \end{enumerate}
    The value of the supremum in (1) is called the diameter of the witness. 
    For instance, this definition appears in \cite{CJMST23}.
    These definitions are easily seen to be equivalent up to adding a factor of $\le 2$ to go between radius and diameter: 
    Given $f : X \to X$ as in the first definition, let $E$ be the equivalence relation whose classes are the fibers of $f$. 
    Given $E$ as in the second, let $T \subseteq X$ be a transversal of $E$ and let $f(x)$ be the unique element of $T \cap [x]_E$. 
    This equivalence also holds up in the Borel context, since $E$ as in this definition will always be smooth, and so admit a Borel transversal. 

This second definition also makes it clear that a colored dimension witness can be converted to an uncolored one, up to adding factors of $\le 2$ to scale and diameter/radius. 
If $X = U_0 \sqcup \cdots \sqcup U_d$ is a colored witness at scale $s$ then one can take $E$ to be the equivalence relation whose classes are the connected components of $G^s \uhr U_i$ for $i \le d$.
Somewhat surprisingly, it is also possible to go in the other direction, at least asymptotically. 

\begin{lem}[Conley--Jackson--Marks--Seward--Tucker-Drob for Borel {\cite[Lemma 3.1]{CJMST23}}]\label{lem:colored_eq_uncolored}
    Let $G$ be a locally finite (Borel) graph on a (standard Borel) space $X$. Let $d \in \N$. The following are equivalent:
    \begin{enumerate}
        \item For every $s \in \N$, $G$ admits a (Borel) uncolored dimension $d$ witness at scale $s$
        \item For every $s \in \N$, $G$ admits a (Borel) colored dimension $d$ witness at scale $s$.
    \end{enumerate}
\end{lem}

Later we will need a quantitative version of this Lemma (\zcref{colored-uncolored-hard}, \zcref{colored-uncolored-soft}). For now, it simply justifies the following definition. 

\begin{defn}[Gromov \cite{Gro93}, Conley--Jackson--Marks--Seward--Tucker-Drob \cite{CJMST23}]\label{defn:borel_asdim}
    If the equivalent conditions of \zcref{lem:colored_eq_uncolored} hold for some locally finite graph $G$, we say $\asdim(G) \le d$. If the Borel versions hold we say $\asdim_B(G) \le d$.
    These parameters are called the \emphd{(Borel) asymptotic dimension} of $G$.  
\end{defn}

For most of this paper we will focus on uncolored dimension witnesses, but in the proof of \zcref{thm:main} we will eventually need to convert to colored witnesses.

The key insight of \cite{CJMST23} was the following connection between Borel asymptotic dimension and hyperfiniteness. 

\begin{thm}[Conley--Jackson--Marks--Seward--Tucker-Drob {\cite[Theorem 1.7]{CJMST23}}]\label{thm:hyperfinite_of_dimension}
    Let $G$ be a locally finite Borel graph. If $\asdim_B(G) < \infty$ then $E_G$ is hyperfinite.
\end{thm}

Let $M$ be a monoid, $R \subseteq M$, and $X$ an $M$-set. Let $G^X_R$ denote the graph on $X$ with an edge between $x$ and $rx$ for every $x \in X$ and $r \in R$. 
This is called the \emphd{Schreier graph} of $X$ with respect to $R$.
If $R,S \Subset M$ both generate $M$, it is not hard to see that $\asdim(G_R^X) = \asdim(G_S^X)$. 
Accordingly, when $M$ is finitely generated we define $\asdim(M \car X)$ to be equal to $\asdim(G_R^X)$ for some/any generating set $R \Subset M$. 
By a \emphd{Borel $M$-space} we mean a standard Borel space $X$ equipped with a Borel action $M\car X$. 
We define $\asdim_B(M \car X)$ similarly when $X$ is a Borel $M$-space. 
Finally we write $\asdim(M) := \asdim(M \car M)$. 

The following LCL encodes uncolored dimension witnesses.

\begin{defn}
  Let $M$ be a monoid,
  let $S, R \Subset M$,
  and let $d \in \N$.
  The $M$-LCL $\Pi_{\dim_{S, R,d}}$
  is the set of $l \in R^S$ such that
  the set $\{l(s) s \mid s \in S\} \subseteq RS$
  has size at most $d + 1$.
\end{defn}

\begin{lem}\label{lem:radius_shift}
    Let $M$ be a monoid, 
    let $S,R \Subset M$,
    let $d \in \N$, and
    let $m \in M$.
    There is a $\{1\}$-local reduction from 
    $\Pi_{\dim_{S,mR,d}}$ to $\Pi_{\dim_{S,R,d}}$.
\end{lem}

\begin{proof}
    Our reduction is the function $r \mapsto mr$ from $R$ to $mR$. For any $l : S \to R$, 
    $|\{ml(s)s \mid s\in S\}| \le |\{l(s) s \mid s \in S\}|$, proving that this is a reduction.
\end{proof}

\begin{lem}\label{lem:dimension_lcl_to_witness_point}
    Let $M$ be a commutative monoid, $S,R \Subset M$ with $1 \in S$, $s,d \in \N$, $X$ an $M$-set, and $x \in X$.
    Suppose $S$ has an lcm $\hat{s}$.
    Let $l : X \to R$ and define $f : X \to X$ by 
    $f(y) = l(\hat{s}^s y) \hat{s}^s y$.
    If $l$ is a
    $\Pi_{\dim_{S^{2s},R,d}}$-labeling at $x$, 
    then $|f[B_{G_S^X}(x,s)]| \le d + 1$.
\end{lem}

\begin{proof}
    We first claim that $\hat{s}^s[B_{G_S^X}(x,s)] \subseteq S^{2s}x$. Induct on $s$. In the base case $s = 0$ there is nothing to show. 
    Now suppose $y \in B_{G_S^X}(x,s+1)$. Let $z \in B_{G_S^X}(x,s)$ be adjacent to $y$, and assume by IH that $\hat{s}^{s}z \in S^{2s}x$. 
    There is some $m \in S$ for which $mz = y$ or $my = z$.
    In the first case $\hat{s}^{s+1} y = \hat{s}^{s+1} mz \in S^{2s} \hat{s} m x \subseteq S^{2s+2} x$. 
    In the second case let $m' \in S$ with $mm' = \hat{s}$.
    Then $\hat{s}^{s+1}y = \hat{s}^s m'my = \hat{s}^s m' z \in S^{2s}m'x \subseteq S^{2s+2} x$.

    Now, letting $g : X \to X$ be given by $g(x) = l(x) x$, it suffices to show that $|g[S^{2s} x]| \le d+1$.
    By hypothesis on $l$, the set $T := \{l(mx)m \mid m \in S^{2s} \}$ has cardinality at most $d + 1$. 
    Now for all $m \in S^{2s}$, $g(mx) = l(mx)mx \in Tx$, so 
    $|g[S^{2s} x]| \le |Tx| \le  d+1$ as desired. 
\end{proof}

\begin{cor}\label{cor:dimension_lcl_to_witness}
    Let $M$ be a commutative monoid, $S\Subset M$ with lcm $\hat{s}$ and $1 \in S$, $s, d, r \in \N$, $X$ an $M$-set.
    If $l : X \to S^r$ is a $\Pi_{\dim_{S^{2s},S^{r},d}}$-labeling of $X$, then the function $f : X \to X$ defined by 
    $f(x) = l(\hat{s}^sx) \hat{s}^s x$ is an uncolored dimension $d$ witness for $G_S^X$ at scale $s$ with radius $r+s$.
\end{cor}

\begin{proof}
    \zcref{lem:dimension_lcl_to_witness_point} shows that $f$ is an uncolored dimension $d$ witness for this graph at scale $s$.
    Furthermore, for all $x \in X$, $f(x) = mx$ for some $m \in S^{r}\hat{s}^s \subseteq S^{r + s}$, proving the radius statement.
\end{proof}

\begin{lem}\label{lem:asdim_local_soft}
    Let $M$ be a finitely generated cancellative commutative monoid.
    There exists $D = O_{\rk(M)}(1)$ such that
    for every $1 \in S \Subset M$ with an lcm, there is $R \Subset M$
    such that $\Pi_{\dim_{S,R,D}} \in \LOCAL_M$.
\end{lem}

\begin{proof}
    We prove a quantitative version of this in \zcref{cor:monoid_asdim_local}.

    Alternatively, it can be derived from known facts with $D = \rk(M)$. Here is a sketch. We may assume $M = \langle S \rangle$.
    Let $\Gamma = M^{\gp}$. 
    By \cite[Theorem 10.7, Lemma 10.3]{CJMST23} 
    for every free 0-dimensional Polish $\Gamma$-space $X$, $G_S^X$ admits continuous uncolored dimension $D$-witnesses at every scale. 

    Since $X$ is free, if $f : X \to X$ is such a witness at scale 1, say with radius $r$, we may define $l : X \to S^{\pm r}$ as the unique function for which $l(x) x = f(x)$ for all $x \in X$. 
    The fact that $f$ is a dimension witness easily translates to the fact that $l$ is a $\Pi_{\dim_{S,S^{\pm r},D}}$-labeling of $X$, and it is continuous since $f$ is. 

    By \zcref{thm:monoid_vs_group}, we conclude that $\Pi_{\dim_{S,S^{\pm r},D}} \in \LOCAL_\Gamma$. 
    Then by \zcref{lem:radius_shift} $\Pi_{\dim_{S,g S^{\pm r}, D}} \in \LOCAL_\Gamma$ for any $g \in \Gamma$. 
    When $g = \hat{s}^r$ for $\hat{s}$ an lcm of $S$ we have $g S^{\pm r} \subseteq M$, 
    and then $\Pi_{\dim_{S,g S^{\pm r}, D}} \in \LOCAL_M$ by \zcref{thm:monoid_vs_group} again.  

\end{proof}
\begin{thm}\label{thm:borel_asdim_fg}
    Let $M$ be a finitely generated commutative monoid,
    and $X$ a free bounded-to-one standard Borel $M$-space.
    Then $\asdim_B(M \car X) \le O_{\rk(M)}(1)$.
\end{thm}

\begin{proof}
    $M$ must be cancellative (unless $X$ is empty) since it acts freely on $X$.
    Let $D = O_{\rk(M)}(1)$ be as in \zcref{lem:asdim_local_soft}.
    Fix a generating set $S \Subset M$ with $1 \in S$ such that $S$ has an lcm. 
    We want to show $\asdim_B(G_S^X) \le D$.
    Fix $s \in \N$.
    By \zcref{lem:asdim_local_soft}
    and \zcref{prop:local_to_borel_global}
    $X$ admits a Borel $\Pi_{\dim_{S^{2s},R,D}}$-labeling, say $l$, for some $R \Subset M$. 
    There is some $r$ for which $R \subseteq S^{r}$, and then $l$ is also a $\Pi_{\dim_{S^{2s}, S^r, D}}$-labeling.
    The result now follows from
    \zcref{cor:dimension_lcl_to_witness}.
\end{proof}

\begin{rmk}
    More generally, if $M$ is a finitely generated commutative monoid and $X$ is a bounded-to-one Borel $M$-space, we can prove $\asdim_B(M \car \cS(X)) \leq O_{\rk(M)}(1)$. 
    The proof is the same, 
    except that one works with free actions of various quotients of $M$
    (as in the proof of the next theorem)
    and needs to know that there is a constant upper bound to the radii of the dimension witnesses provided by \zcref{lem:asdim_local_soft} across all such quotients. 
    This is indeed the case, as we will see in \zcref{cor:monoid_asdim_local}. 
\end{rmk}

We can now prove the main theorem in the case where the monoid is finitely generated.

\begin{thm}\label{thm:main_fg}
    Let $M$ be a finitely generated commutative monoid and $X$ a bounded-to-one Borel $M$-space. Then $E_M^X$ is hyperfinite.
\end{thm}

\begin{proof}
  We may assume $\sim_x$ is constant for $x \in \cS(X)$ since this is $E_M^{\cS(X)}$-invariant by the definition of $\cS(X)$, there are only countably many options for $\sim_x$ since $M$ is Noetherian, and the assignment $x \mapsto \sim_x$ is clearly Borel. Call this congruence $\sim$ and let $M' = M/\sim$.

  Now we have a free action $M' \car \cS(X)$ by \zcref{cor:free_action_on_stable}, and it is still Borel and bounded-to-one, so by \zcref{thm:borel_asdim_fg} $\asdim_B(M' \car \cS(X)) < \infty$,
  so by \zcref{thm:hyperfinite_of_dimension} $E_{M'}^{\cS(X)}$ is hyperfinite.
  Since $\cS(X)$ is a complete section of $E_M^X$
  and $E_{M'}^{\cS(X)} = E_M^{\cS(X)} = E_M^X \uhr \cS(X)$, 
  we are done by \cite[Theorem 1.3(vi)]{jkl}.
\end{proof}

\section{Smooth separation index}\label{sec:ssi}

This section will not be needed for the proof of \zcref{thm:main}
(but see \zcref{qn:classical_asdim} and the surrounding discussion).
Instead, it contains some interesting combinatorial consequences
of our work in the finitely generated case. 

To begin,
we will need a new parameter inspired by the asymptotic separation index of \cite{CJMST23}:

\begin{defn}[Conley--Jackson--Marks--Seward--Tucker-Drob \cite{CJMST23}]\label{def:asi}
    Let $G$ be a locally finite graph on a vertex set $X$. Let $s,d \in \N$.
    \begin{itemize}
        \item An \emphd{uncolored separation index $d$ witness at scale $s$} for $G$ is an equivalence relation $E \subseteq E_G$ on $X$ such that 
        \begin{enumerate}
            \item Each $C \in X/E$ is finite. 
            \item For each $x \in X$, $B_G(x,s)$ meets at most $d+1$ many $E$-classes.
        \end{enumerate}

        \item A \emphd{colored separation index $d$ witness at scale $s$} for $G$ is a partition $X = U_0 \sqcup \cdots \sqcup U_d$ such that for each $i \le d$, each connected component of $G^s \uhr U_i$ is finite.
    \end{itemize}
\end{defn}

The abstract existence of separation index witnesses is uninteresting. 
$G$ admits a (colored or uncolored) separation index 0 witness at scale $1$ if and only if each connected component of $G$ is finite.
If $x_0$ is a vertex in a locally finite connected graph $G$ and $s \in \N$, the sets $U_i = \{x \mid \rho_G(x_0,x) \in [is, (i+1)s) \text{ mod } 2s\}$ for $i = 0,1$ define a colored separation index 1 witness at scale $s$, and as with dimension, taking the components of the $U_i$'s in a colored witness gives an uncolored witness at half the scale. 

Instead, these witnesses become interesting when $G$ is a Borel graph and we ask for the witnesses to be Borel in the natural sense (as in \zcref{defn:borel_dim}). Thus we restrict our discussion to the Borel setting.
The following are analogues of \zcref{lem:colored_eq_uncolored} and \zcref{defn:borel_asdim}:

\begin{lem}[Conley--Jackson--Marks--Seward--Tucker-Drob {\cite[Lemma 3.1]{CJMST23}}]\label{lem:colored_eq_uncolored_asi}
    Let $G$ be a locally finite Borel graph. Let $d \in \N$. The following are equivalent:
    \begin{enumerate}
        \item For every $s \in \N$, $G$ admits a Borel uncolored separation index $d$ witness at scale $s$.
        \item For every $s \in \N$, $G$ admits a Borel colored separation index $d$ witness at scale $s$.
    \end{enumerate}
\end{lem}

\begin{defn}
    If (1) or equivalently (2) from \zcref{lem:colored_eq_uncolored_asi} hold for some locally finite Borel graph $G$, 
    we say $\asi_B(G) \le d$. 
    This parameter is called the \emphd{Borel asymptotic separation index} of $G$. 
\end{defn}

Clearly dimension witnesses are separation index witnesses at the same scale, so $\asi_B(G) \le \asdim_B(G)$ for all $G$. In fact something much stronger is true:

\begin{lem}[Conley--Jackson--Marks--Seward--Tucker-Drob {\cite[Theorem 4.8(a)]{CJMST23}}]\label{lem:asi_1_of_asdim}
    Let $G$ be a locally finite Borel graph. If $\asdim_B(G) < \infty$ then $\asi_B(G) \le 1$. 
\end{lem}

It is an open question whether there is any locally finite Borel graph $G$ with $1 < \asi_B(G) < \infty$. 

We now introduce our new parameter. 
In applications involving $\asi_B$, often all that is used about an uncolored separation index witness $E$ is that it is smooth. This fact motivates the following weakening.

\begin{defn}
    Let $G$ be a locally finite Borel graph on a standard Borel space $X$. Let $s,d \in \N$.
    \begin{itemize}
        \item A \emphd{Borel uncolored smooth separation index $d$ witness at scale $s$} for $G$ is a Borel equivalence relation $E \subseteq E_G$ such that
        \begin{enumerate}
            \item $E$ is smooth.
            \item For each $x \in X$, $B_G(x,s)$ meets at most $d+1$ many $E$-classes.
        \end{enumerate}
        \item A \emphd{Borel colored smooth separation index $d$ witness at scale $s$} for $G$ is a Borel partition $X = U_0 \sqcup \cdots \sqcup U_d$ such that for each $i \le {d}$, 
        $E_{G^s \uhr U_i}$ is smooth. 
    \end{itemize}
\end{defn}

\begin{rmk}
    There is an alternative definition of uncolored smooth separation index witness more parallel to the one we use for dimension; we could say that a Borel smooth separation index $d$ witness at scale $s$ is a Borel function $f : X \to X$ such that 
    \begin{enumerate}
        \item $f \subseteq E_G$.
        \item For each $x \in X$, $|f[B_G(x,s)]| \leq d+1$.
    \end{enumerate}
    The proof that this is equivalent to our official definition is essentially the same as the analogous fact for dimension. 
\end{rmk}

\begin{lem}\label{lem:colored_eq_uncolored_ssi}
    Let $G$ be a locally finite Borel graph. Let $d \in \N$. The following are equivalent:
    \begin{enumerate}
        \item For every $s \in \N$, $G$ admits a Borel uncolored smooth separation index $d$ witness at scale $s$.
        \item For every $s \in \N$, $G$ admits a Borel colored smooth separation index $d$ witness at scale $s$.
    \end{enumerate}
\end{lem}

\begin{proof}
    Essentially same as the proof of \zcref{lem:colored_eq_uncolored_asi}: As with dimension and separation index, if $U_0,\ldots,U_d$ give a colored witness at scale $s$, then letting $E$ be the equivalence relation whose classes are the components of the $G^s \uhr U_i$'s gives an uncolored witness at half the scale. Note here that $E$ is smooth since $E \cong \bigsqcup_i E_{G^s \uhr U_i}$. 

    For the reverse direction, given an uncolored witness $E$ at scale $(d+1)s$, the proof of the corresponding direction of \zcref{lem:colored_eq_uncolored_asi} (\cite[Lemma 3.1]{CJMST23}) (see also the proof of \zcref{colored-uncolored-hard}) gives a Borel partition $X = U_0 \sqcup \cdots \sqcup U_d$ such that each component of each $G^s \uhr U_i$ meets at most $i+1$ many $E$-classes. Thus $E_{G^s \uhr U_i}$ is finite index over the equivalence relation $E \cap E_{G^s \uhr U_i}$. The latter is smooth since $E$ is, 
    and thus so is the former by, for example, \cite[Proposition 1.6.4]{clemens2016smooth}.
\end{proof}

\begin{defn}\label{defn:assi}
    If (1) or equivalently (2) from \zcref{lem:colored_eq_uncolored_ssi} holds for some locally finite Borel graph $G$, we say $\assi_B(G) \le d$. 
    This parameter is called the \emphd{Borel asymptotic smooth separation index of $G$}.
\end{defn}

\begin{lem}\label{lem:assi_disjoint_union}
    Let $G$ be the disjoint union of countably many locally finite Borel graphs $G_n$. Let $d \in \N$. If $\assi_B(G_n) \leq d$ for each $n$, then $\assi_B(G) \leq d$. 
\end{lem}

\begin{proof}
    Clear from the fact that a countable disjoint union of smooth Borel equivalence relations is smooth. 
\end{proof}

Finite $\asi_B$ and $\assi_B$ are equivalent:

\begin{lem}\label{lem:finite_ssi_eq_finite_asi}
    Let $G$ be a locally finite Borel graph. Then $\assi_B(G) \le \asi_B(G) \le 2\assi_B(G)+1$.  
\end{lem}

\begin{proof}
    The first inequality is clear since Borel equivalence relations with finite classes are smooth by, for example, \cite[Proposition 1.6.3]{clemens2016smooth}.
    For the second, suppose $\assi_B(G) \le d \in \N$.
    Let $s \in \N$ and let $E$ be a Borel uncolored smooth separation index $d$ witness at scale $s$ for $G$.
    Let $X$ be the vertex set of $G$ and $T \subseteq X$ a Borel transversal of $E$.
    Finally, let $F \subseteq E$ be the equivalence relation whose classes are those of the form 
    $\{x \in C \mid \rho_G(x,x_0) \in [2ns, 2(n+1)s)\}$ for $x_0 \in T$, $C = [x_0]_E$, and $n \in \N$. 
    $F$ has finite classes since $G$ is locally finite.
    For any $x \in X$, $B_G(x,s)$ meets at most two $F$-classes within any given $C \in X/E$. 
    Then, since it meets at most $d+1$ $E$-classes by assumption, it meets at most $2(d+1)$ $F$-classes total, as desired. 
\end{proof}

For an example where $\asi_B$ and $\assi_B$ differ, just take any $G$ with an infinite component and $E_G$ smooth.
Then $\asi_B(G) = 1$ but $\assi_B(G) = 0$. 

As was stated previously, our motivation for introducing smooth separation index is that most results about asymptotic separation index use only the fact that Borel equivalence relations with finite classes are smooth, and therefore hold more generally for smooth separation index. 
Smoothness is usually used in these results to conclude that some combinatorial task which can be accomplished abstractly can actually be done in a Borel way. See e.g. \cite[Theorem 5.23]{pikhurko} for a formal statement along these lines. 
The following are examples of statements proven in \cite{CJMST23} for $\asi_B$ in this way which therefore hold also for $\assi_B$ by the same proofs. (Alternatively the first follows from the $\asi_B$ version and \zcref{lem:finite_ssi_eq_finite_asi}.)

\begin{lem}[Essentially {\cite[Theorem 4.2]{CJMST23}}]\label{lem:asdim_of_finite_ssi}
    Let $G$ be a locally finite Borel graph. If $\assi_B(G) < \infty$ then $\asdim_B(G) = \asdim(G)$. 
\end{lem}

\begin{lem}[Essentially {\cite[Corollary 8.2]{CJMST23}}]\label{lem:chromatic_number_of_ssi}
    Let $G$ be a locally finite Borel graph which admits a (not necessarily Borel) proper $k$-coloring. 
    Then $G$ admits a Borel proper $[k \cdot (\assi_B(G) + 1) - \assi_B(G)]$-coloring. 
\end{lem}

A function $f$ between graphs $G$ and $H$ is called \emphd{Lipschitz} if $\sup_{(x,y) \in G} \rho_H(f(x),f(y)) < \infty$. 
Note that this implies $f$ is \emphd{bornologous}: for every $s \in \N$ there exists $s' \in \N$
such that $\rho_G(x,y) \le s \implies \rho_H(f(x),f(y)) \le s'$.

The reason we need to generalize from $\asi_B$ to $\assi_B$ is that the following fact fails for $\asi_B$.

\begin{prop}\label{prop:ssi_pullback}
    Let $G, H$ be locally finite Borel graphs on standard Borel spaces $X,Y$ respectively. Let $f : X \to Y$ be a countable-to-one bornologous Borel function. 
    Then $\assi_B(G) \le \assi_B(H)$. 
\end{prop}

\begin{proof}
    Suppose $\assi_B(H) = d \in \N$.
    Let $s \in \N$. Let $s' \in \N$ such that $\rho_G(x,y) \le s \implies \rho_H(f(x),f(y)) \le s'$. 
    Let $E$ be a Borel uncolored smooth separation index $d$ witness at scale $s'$ for $H$. 
    Let $E':= (f \times f)\inv(E)$. 
    We claim $F := E' \cap E_G$ is a Borel uncolored smooth separation index $d$ witness at scale $s$ for $G$. 
    
    $E'$ is smooth as the pullback of a smooth equivalence relation, and it has countable classes since $f$ is countable-to-one. 
    Therefore $F$ is smooth since smoothness is closed downwards for CBERs \cite[Proposition 1.6.2]{clemens2016smooth}.
    For any $x \in X$, $f[B_G(x,s)] \subseteq B_H(f(x),s')$, so this image meets at most $d+1$ $E$-classes, so $B_G(x,s)$ meets at most $d+1$ $(f \times f)\inv(E)$-classes. It is also contained in a single $E_G$ class, so it meets at most $d+1$ $F$-classes. 
\end{proof}

As with dimension, when $M$ is a finitely generated monoid and $X$ is a Borel $M$-set, $\asi_B$ and $\assi_B$ for the Schreier graphs of $X$ are independent of the generating set we take for $M$, so denote them as $\asi_B(M \car X)$ and $\assi_B(M \car X)$ respectively. 

\begin{cor}\label{cor:monoid_ssi_1}
    Let $M$ be a finitely generated commutative monoid and $X$ a bounded-to-one Borel $M$-set. 
    Then $\assi_B(M \car X) \le 1$. 
\end{cor}

\begin{proof}
    As in the proof of \zcref{thm:main_fg} and using \zcref{lem:assi_disjoint_union}, we may assume $\sim_x$ is constant for $x \in \cS(X)$, and then by \zcref{cor:free_action_on_stable} and \zcref{thm:borel_asdim_fg} we have $\asdim_B(M \car \cS(X)) < \infty$. 
    Then by \zcref{lem:asi_1_of_asdim, lem:finite_ssi_eq_finite_asi} we have 
    $\assi_B(M \car \cS(X)) \leq 1$. 
    
    Fix a finite generating set $R \Subset M$ with lcm $\hat r$.
    We now want to apply \zcref{prop:ssi_pullback} to the map $T: X \to \cS(X)$ from \zcref{defn:go_to_stable}.
    We claim $T$ is Lipschitz with respect to $G_R^X$ and $G_R^{\cS(X)}$. 
    This follows easily from \zcref{prop:stable_born}(2). 
    It is also easy to see that the map \(T\) is Borel, and it is countable-to-one since $T \subseteq E_M^X$. 
\end{proof}

The large body of research on Borel combinatorics under the assumption of $\asi = 1$ then easily adapts to give statements like the following for commutative monoid actions.

\begin{cor}\label{cor:borel_3_col}
    Let $M$ be a commutative monoid, $S \Subset M$, and $X$ a bounded-to-one Borel $M$-set.
    If $G_S^X$ has a $k$-coloring then it has a Borel $(2k-1)$-coloring.

    In particular, if $X$ is a free bounded-to-one Borel $\N^d$-set and $S = \{e_0,\ldots,e_{d-1}\}$ denotes the usual generating set, then $G_S^X$ admits a Borel $3$-coloring. 
\end{cor}

\begin{proof}
    The first statement is immediate from
    \zcref{cor:monoid_ssi_1,
    lem:chromatic_number_of_ssi}.

    For the second statement, it suffices to show that $G_S^X$ is bipartite.
    By working one component at a time, we may assume $X$ is transitive. 
    Fix $x_0 \in X$.
    For any $x \in X$, there exist $g,h \in \N^d$ such that $gx = hx_0$. 
    We claim that $h - g = h' - g'$ for any other choice of $g',h' \in \N^d$ with $g'x = h'x_0$. 
    Indeed, $(g + h')x_0 = (g+g')x = (h + g')x_0$, so $g+h' = h+g'$ by freeness. 
    We thus define $\phi:X \to \Z^d$ by setting $\phi(x) = h-g$ for any/all such $g,h$.

    We now claim that $\phi$ is a homomorphism from $G_S^X$ to $G_S^{\Z^d}$, the usual Cayley graph of $\Z^d$. 
    It suffices to show that it is $\N^d$-equivariant. 
    Let $k \in \N^d$ and $x \in X$. 
    Let $g,h \in \N^d$ with $gx = hx_0$. 
    Then $(g + k)x = (h + k)x_0$, so $\phi(kx) = h + k - g = \phi(x) + k$. 
\end{proof}

Other results about Borel colorings of Schreier graphs of $\N^d$ actions with respect to the generating set $\{e_0,\ldots,e_{d-1}\}$ were obtained by Palamourdas \cite{palamourdas}. 

We end this section by noting the following combination of
\zcref{cor:monoid_ssi_1, lem:asdim_of_finite_ssi}:

\begin{cor}\label{cor:comm_monoid_asdim_eq}
    Let $M$ be a finitely generated commutative monoid and $X$ a bounded-to-one Borel $M$-set. 
    Then $\asdim_B(M \car X) = \asdim(M \car X)$. 
\end{cor}

However, even in the classical setting, we do not have a full understanding of the dimension of non-free commutative monoid actions. 

\begin{qn}\label{qn:classical_asdim}
    Let $M$ be a finitely generated commutative monoid
    and $X$ a (bounded-to-one) $M$-set.
    Is $\asdim(M \car X) < \infty$?
    Is $\asdim(M \car X) \le \rk(M)$? 
\end{qn}

The first question here was \zcref{qn:intro_classical_dimension} in the introduction. 
A positive answer to this question would yield a very clean proof of the non-finitely-generated case of \zcref{thm:main} thanks to \zcref{cor:comm_monoid_asdim_eq} and the following generalization of \zcref{thm:hyperfinite_of_dimension}:

\begin{thm}[Conley--Jackson--Marks--Seward--Tucker-Drob {\cite[Theorem 1.10]{CJMST23}}]\label{thm:hyperfinite_of_dimension_union}
    Let $G_0 \subseteq G_1 \subseteq \cdots$ be an increasing sequence of locally finite Borel graphs such that for each $n$, $\asdim_B(G_n) < \infty$. 
    Then $\bigcup_n E_{G_n}$ is hyperfinite.
\end{thm}

Unfortunately we were not able to answer this question. 
Instead, we managed to get good enough quantitative bounds on the existence of dimension witnesses at various scales to still conclude hyperfiniteness. 
The details of this argument constitute the remainder of the paper.

\section{Quantitative refinements of asymptotic dimension}
\label{sec:quant_asdim}

We begin with more precise quantitative versions of the statements from \zcref{sec:asdim}, starting with \zcref{lem:asdim_local_soft}. 


\begin{lem}\label{lem:doubling}
  For every $d \in \N$,
  there exists $D \in \N$ such that
  for every abelian group $\Gamma$,
  every $A \Subset \Gamma$ with $|A| \le d$,
  and every $s \in \N$,
  we have $|A^{4s}| \le D|A^s|$. 
\end{lem}
\begin{proof}
 If $A=\varnothing$, there is nothing to prove.  Choose $a_0\in A$ and put $B=a_0^{-1}A$.  Since $\Gamma$ is abelian, for every $t\in\N$ we have $B^t=a_0^{-t}A^t$, and hence $|B^t|=|A^t|$.  Thus, replacing $A$ by $B$, we may assume that $1\in A$.

  Write $A=\{1,a_1,\ldots,a_q\}$, where $q\le d-1$.  If
  $q=0$, the assertion is immediate.  Assume $q\ge1$, and let
  $\varphi:\N^q\to\Gamma$ be given by
  $\varphi(n_1,\ldots,n_q)=a_1^{n_1}\cdots a_q^{n_q}$.  If
  \[
    \Sigma_t:=\{u\in\N^q:\|u\|_1\le t\},
  \]
  then $A^t=\varphi(\Sigma_t)$.

  We claim that $\Sigma_{4s}$ can be covered by a number, depending only
  on $q$, of translates of $\Sigma_s$.  When $s<q$, this follows from
  the crude bound $|\Sigma_{4s}|\le(4q+1)^q$.  Suppose $s\ge q$, and put
  $h=\lfloor s/q\rfloor$.  For $u=(u_i)_{i<q}\in\Sigma_{4s}$, set
  $v_i=h\lfloor u_i/h\rfloor$.  Then
  $u-v\in\Sigma_s$, while each $v_i/h$ lies in
  $\{0,1,\ldots,8q\}$.  Thus at most $(8q+1)^q$ translation vectors
  $v$ occur.  Applying $\varphi$ to such a covering gives
  \[
    A^{4s}=\varphi(\Sigma_{4s})
      \subseteq \bigcup_{j=1}^{D} \varphi(v_j)A^s,
  \]
  with $D=D(d)$, and the cardinality estimate follows.
\end{proof}

The following argument is essentially the same one used in \cite{CJMST23} to establish finite Borel asymptotic dimension of actions of polynomial growth groups.  

\begin{prop}
    Let $\Gamma$ be an abelian group, $1 \in S \Subset \Gamma$ of cardinality $d \in \N$, and $s \in \N$. There exist $D \in O_{d}(1)$ and $C \in \poly_d(s)$ such that
    \[
      \Pi_{\dim_{S^{\pm s}, S^{\pm 2s}, D}}
      \in \LOCAL_{S^{\pm 1}}(C \log^*(n)).
    \]
\end{prop}

\begin{proof}
    Fix $d$ and choose $D$ as in \zcref{lem:doubling} with respect to $2d$.
    By work of Linial \cite{linial}, 
    there is a $O(1) \log^*n + \poly(\Delta)$-round local algorithm for finding a maximal independent set in graphs of maximum degree $\Delta \in \N$. 
    (Linial's result is about $\Delta+1$-coloring, but a maximal independent set can easily be obtained from such a coloring greedily in $\Delta+1$ additional rounds.)

    Consider the Cayley graph $G := G_{S^{2s}}^\Gamma$. $G$ is regular with degree $|S^{\pm 2s}| - 1 \le (4s+1)^d$, so, letting $\Pi$ be the $\Gamma$-LCL whose solutions are maximal $G$-independent sets (identified with their characteristic functions), 
    \begin{align*}
        \Pi \in \LOCAL_{S^{\pm 2s}}(O(1) \log^*(n) + \poly_d(s)) \subseteq \LOCAL_{S^{\pm 1}} (2s[(O(1) \log^*(n) + \poly_d(s))]) \subseteq \\ \LOCAL_{S^{\pm 1}} (\poly_d(s) \log^* n).
    \end{align*}
    We now give a $S^{\pm 3s}$-local reduction from $\Pi_{\dim_{S^{\pm s}, S^{\pm 2s}, D}}$ to $\Pi$, finishing the proof by \zcref{lem:local_reduction_additive}. 
    Let $<$ be a linear order on $S^{\pm 2s}$. Our reduction will be the local algorithm 
    $B : 2^{S^{\pm 3s}} \to S^{\pm 2s}$ defined by 
    \[ B(l) = \begin{cases} \min_< (l\inv(1) \cap S^{\pm 2s}) & \text{if } l\inv(1) \cap S^{\pm 2s} \neq \emptyset \\ 
                            1 & \text{otherwise.}
              \end{cases}\]
    To prove $B$ is a reduction between the claimed LCLs, suppose $l : \Gamma \to 2$ is a $\Pi$-labeling of $S^{\pm 3s}$. 
    By maximality, for all $g \in S^{\pm s}$ there is $h \in S^{\pm 2s}$ such that $l(hg) = 1$. Then $(B * l)(g)$ is the $<$-least such $h$. 
    We need to show $|\{ (B * l)(g) \cdot g \mid g \in S^{\pm s} \}| \le D$.
    By the above, this is a subset of $F :=l\inv(1) \cap S^{\pm 3s}$.
    Since $l$ is a $\Pi$-labeling of $S^{\pm 3s}$, $F$ is $G$-independent. Thus the sets $S^{\pm s} \cdot g$ for $g \in F$ are pairwise disjoint. These are contained in $S^{\pm 4s}$ and have cardinality $|S^{\pm s}|$, so by choice of $D$ we get $|F| \le D$. 
\end{proof}

\begin{cor}\label{cor:monoid_asdim_local}
   Let $M$ be a cancellative commutative monoid, $1 \in S \Subset M$ of cardinality $d \in \N$ with an lcm, and $s \in \N$. There exist $D \in O_{d}(1)$ and $C \in \poly_d(s)$ such that
    \[\Pi_{\dim_{S^s, S^{4s}, D}} \in \LOCAL_S(C \log^*(n)).\]
\end{cor}

\begin{proof}
    Let $\hat s$ be an lcm of $S$.
    We work in $M^{\gp}$.
    By \zcref{lem:rounds_monoid_vs_rounds_group}, it suffices to show that this LCL is in $\LOCAL_{S^{\pm 1}}(C \log^*(n))$ for such a $D$ and $C$. 
    The previous proposition gives this for $\Pi_{\dim_{S^{\pm s}, S^{\pm 2s}, D}}$.
    Replacing the first $S^{\pm s}$ with $S^s$ only makes the LCL easier, 
    and then since $\hat s^{2s}S^{\pm 2s} \subseteq S^{4s}$
    we have a $\{1\}$-local reduction from our desired LCL to this one by \zcref{lem:radius_shift}. 
\end{proof}

We now turn to quantitative versions of \zcref{lem:colored_eq_uncolored}. 

\begin{lem}\label{colored-uncolored-hard}
  Let $G$ be a locally finite Borel graph
  and $d, s, r \in \N$.
  \begin{enumerate}[label=(\arabic*)]
    \item\label{colored-to-uncolored-hard-easy-dir}
      If $G$ has a Borel colored
      dimension $d$ witness
      at scale $2s$
      with diameter $r$,
      then it has a Borel uncolored
      dimension $d$ witness
      at scale $s$
      with radius $r$.
    \item\label{uncolored-to-colored-hard-hard-dir}
      If $G$ has a Borel uncolored
      dimension $d$ witness
      at scale $(d+1)s$
      with radius $r$,
      then it has a Borel colored
      dimension $d$ witness
      at scale $s$
      with diameter $2r + ds$.
  \end{enumerate}
\end{lem}
\begin{proof}
  Let $X$ be the vertex set of $G$.
  \begin{enumerate}[label=(\arabic*)]
    \item 
      See the two paragraphs preceding \zcref{lem:colored_eq_uncolored}. 
    \item
      Let $f : X \to X$
      be a Borel uncolored
      dimension $d$ witness
      at scale $(d+1)s$
      with radius $r$.
      For every $i \le d$,
      define the Borel set
      $U_i := \{x \in X : f[B(x, is)] = f[B(x, (i+1)s)]\}$.
      We claim that the family $(U_i)_{i \le d}$
      is a Borel colored
      dimension $d$ witness
      at scale $s$
      with diameter $2r + ds$.
    
      First of all,
      note that this is a cover,
      since for every $x \in X$,
      the increasing sequence
      $(f[B(x, is)])_{i \le d+1}$
      consists of $d+2$ sets of cardinality in $[1, d+1]$,
      so there must be two adjacent terms which are equal.
    
      Fix $i \le d$.
      Note that for $x, x' \in U_i$ with $\rho(x, x') \le s$,
      we have
      \[
        f[B(x, is)] 
        \subseteq f[B(x', (i+1)s)]
        = f[B(x', is)]
      \]
      and vice versa,
      so $f[B(x, is)] = f[B(x', is)]$.
      So the transitive closure of $\{(x, x') \in (U_i)^2 : \rho(x, x') \le s\}$
      is contained in the equivalence relation
      $E := \{(x, x') \in X^2 : f[B(x, is)] = f[B(x', is)]\}$.
      Now if $x E x'$,
      then setting $y := f(x)$,
      we have $y \in f[B(x, is)] = f[B(x', is)]$,
      so fixing $x'' \in B(x', is)$ with $f(x'') = y$,
      we have
      \begin{align*}
        \rho(x, x')
        & \le \rho(x, y) + \rho(y, x'') + \rho(x'', x') \\
        & \le r + r + is
      \end{align*}
      So every $E$-class has diameter $\le 2r + ds$.\qedhere
  \end{enumerate}
\end{proof}

\begin{defn}
  Let $\mc D$ and $\mc R$ be sets of non-decreasing functions $\N \to \N$,
  both closed under pointwise $\leq$. 
  We say that the pair $(\mc D, \mc R)$ is \emphd{linearly closed}
  if $\mc D$ is closed under precomposition with $O(n)$,
  and $\mc R$ is closed under composition on both sides with $O(n)$ and addition with $O(n)$.
\end{defn}

\begin{defn}
  Let $G$ be a locally finite (Borel) graph.
  \begin{enumerate}[label=(\roman*)]
    \item
    Given $d, r : \N \to \N$,
    we say that $G$ has (Borel) uncolored
    $(d, r)$-dimension growth
    if for all $s \in \N$,
    $G$ has a (Borel) uncolored
    dimension $d(s)$ witness
    at scale $s$
    with diameter $r(s)$.
  \item
    Let $(\mathcal{D}, \mathcal{R})$ be linearly closed.
    We say that $G$ has (Borel) uncolored
    $(\mathcal{D}, \mathcal{R})$-dimension growth
    if there exist $d \in \mathcal{D}$ and $r \in \mathcal{R}$
    such that $G$ has (Borel) uncolored $(d, r)$-dimension growth.
  \end{enumerate}
  We define these notions for colored dimension growth analogously. 
\end{defn}
Note that if $(\mathcal{D}, \mathcal{R})$ is linearly closed,
then both uncolored and colored
$(\mathcal{D}, \mathcal{R})$-dimension growth
are closed under (Borel) quasi-isometric embeddings,
i.e. if $X$ has a (Borel) quasi-isometric embedding into $Y$,
then if $Y$ has (Borel) uncolored/colored
$(\mathcal{D}, \mathcal{R})$-dimension growth,
then so does $X$.
In particular,
the notions of (Borel) uncolored/colored
$(\mathcal{D}, \mathcal{R})$-dimension growth
are well-defined for Borel actions of finitely generated monoids
in that they do not depend on the choice of generating set. 

Note that for every $k \in \N$,
having asymptotic dimension $< k$ is the same as having
$(\text{bounded by $k$}, \text{all functions})$-dimension growth,
and having Assouad--Nagata \cite{ANdim} dimension $< k$ is the same as having
$(\text{bounded by $k$}, O(n))$-dimension growth.
For this paper,
we will be concerned with 
$(O(\log n), \poly)$-dimension growth.

\begin{prop}\label{colored-uncolored-soft}
  Let $(\mc D, \mc R)$ be linearly closed,
  and suppose that the following hold:
  \begin{enumerate}[label=(\roman*)]
    \item
      for every $d \in \mc D$,
      there is some $d' \in \mc D$
      such that for all $s \gg 0$,
      we have $d(s (d'(s) + 1)) \le d'(s)$;
    \item
      for every $r \in \mc R$ and $d \in \mc D$,
      the function $s \mapsto r(s(d(s) + 1))$ is in $\mc R$.
  \end{enumerate}
  
  Then a locally finite (Borel) graph
  has (Borel) uncolored $(\mc D, \mc R)$-dimension growth
  iff it has (Borel) colored $(\mc D, \mc R)$-dimension growth.
\end{prop}

Thus when $(\mc D, \mc R)$ satisfy these hypotheses,
we can omit the modifier ``(un)colored'',
and we will just refer to
\emphd{(Borel) $(\mc D, \mc R)$-dimension growth}.
In particular,
we will do this in the case when
$(\mc D, \mc R)$ is $(O(\log n), \poly)$.
Note that for every $k \in \N$,
the pairs $(\text{bounded by $k$}, \text{all functions})$
and $(\text{bounded by $k$}, O(1))$
also satisfy these hypotheses. The former recovers \zcref{lem:colored_eq_uncolored}. 

\begin{proof}
  Let $G$ be a locally finite Borel graph.

  \begin{itemize}
    \item[($\implies$)]
      Fix $d \in \mc D$ and $r \in \mc R$
      so that $G$ has Borel colored $(d, r)$-dimension growth.
      For all $s$,
      there is a Borel colored
      dimension $d(2s)$-witness
      at scale $2s$
      with radius $r(2s)$.
      Thus by \zcref{colored-uncolored-hard},
      for all $s$,
      there is a Borel uncolored
      dimension $d(2s)$ witness
      at scale $s$
      with radius $r(2s)$.
    \item[($\impliedby$)]
      Fix $d \in \mc D$ and $r \in \mc R$
      so that $G$ has Borel uncolored $(d, r)$-dimension growth.
      Let $d' \in \mc D$ such that for all $s \gg 0$,
      we have $d(s (d'(s) + 1)) \le d'(s)$.
      Now for $s \gg 0$,
      there is a Borel uncolored
      dimension $d(s (d'(s) + 1))$ witness
      at scale $s (d'(s) + 1)$
      with radius $r(s (d'(s) + 1))$.
      Since $d(s (d'(s) + 1)) \le d'(s)$,
      by \zcref{colored-uncolored-hard},
      there is a Borel colored
      dimension $d'(s)$ witness
      at scale $s$
      with radius $2r(s (d'(s) + 1)) + sd'(s)$.
  \end{itemize}
\end{proof}

Finally, we turn to generalizations of
\zcref{thm:hyperfinite_of_dimension, thm:hyperfinite_of_dimension_union}
(\cite[Theorems 1.7 and 1.10]{CJMST23}).
The most general result here is the following.

\begin{thm}[Greb\'ik--Marks--Rozho\v{n}--S. {\cite[Theorem C]{GMRS26}}]\label{increasing-union-full}
  Let $(G_m)_{m \in \N}$
  be an increasing sequence of locally finite Borel graphs.
  If there are sequences $(a_n)_n, (b_n)_n$ in $\N$
  with $a_n \le b_n$ and $a_n \to \infty$
  such that for all $m \in \N$,
  for all $n \gg 0$,
  the graph $G_m$ admits a Borel colored
  dimension $\floor{a_{n-1}/(24b_{n-2})}$ witness 
  at scale $a_n$ with diameter $b_n$,
  then $\bigcup_m E_{G_m}$ is hyperfinite.
\end{thm}

\begin{cor}\label{increasing-union-simple}
  Let $(\mathcal{D}, \mathcal{R})$ be linearly closed,
  and suppose that
  for every $d \in \mathcal{D}$,
  for every $r \in \mathcal{R}$,
  for $x \gg 0$,
  we have $d(r(x^2)^2) \le x$.
  Then for every increasing sequence $(G_m)_{m \in \N}$
  of locally finite Borel graphs of Borel colored
  $(\mathcal{D}, \mathcal{R})$-dimension growth,
  the equivalence relation $\bigcup_{m \in \N} E_{G_m}$
  is hyperfinite.
\end{cor}
We will use this when $(\mathcal{D}, \mathcal{R})$
is $(O(\log n), \poly)$.
Note that when $(\mathcal{D}, \mathcal{R})$
is $(O(1), \text{all functions})$,
we recover \zcref{thm:hyperfinite_of_dimension_union}. 

\begin{proof}
  We will apply \zcref{increasing-union-full} with $a_n = 24 b_{n-1}^2$,
  so it suffices to define a sequence $(b_n)_n$ in $\N$
  such that for all $m \in \N$,
  for all $n \gg 0$,
  the graph $G_m$ has a Borel colored dimension $b_n$ witness
  at scale $24 b_{n+1}^2$ with diameter $b_{n+2}$.

  We claim that there are functions $d, r : \N \to \N$
  satisfying $\forall x \in \N \; [d(24 r(24 x^2)^2) \le x]$
  such that for all $m \in \N$,
  for all $s \gg 0$,
  the graph $G_m$ has a Borel colored dimension $d(s)$ witness
  at scale $s$ with diameter $r(s)$.
  Fix $(d_m)_m \subseteq \mathcal{D}$
  and $(r_m)_m \subseteq \mathcal{R}$
  such that for all $m \in \N$,
  and that for all $s \gg 0$,
  $G_m$ has a Borel colored dimension $d_m(s)$ witness
  at scale $s$ with diameter $r_m(s)$.
  By replacing $d_m$ with $\max_{i \le m} d_i$,
  and similarly for $r_m$,
  we can assume that the sequences $(d_m)_m$ and $(r_m)_m$
  are non-decreasing.
  Given $x \in \N$,
  let $i$ be the largest integer in $[0, x]$
  such that for all $y \ge x$,
  we have $d_i(24 r_i(24 y^2)^2) \le y$,
  then set $r(x) := r_i(x)$
  (set $r(x) = 0$ if there is no such integer $i$).
  Note that for all $m \in \N$,
  for all $x \gg 0$,
  we have $d_m(24 r(24 x^2)^2) \le x$.
  Next,
  given $x \in \N$,
  let $j$ be the largest integer in $[0, x]$
  such that for all $y \in \N$ with $24 r(24 y^2)^2 \ge x$,
  we have $d_j(24 r(24 y^2)^2) \le y$,
  then set $d(x) := d_j(x)$
  (set $d(x) = 0$ if there is no such integer $j$).
  This yields the desired $d$ and $r$.

  Now define the sequence $(b_n)_n$ by
  $b_0 = 2$ and $b_{n+1} = r(24 b_n^2)$.
  Then for all $m \in \N$,
  for all $n \gg 0$,
  the graph $G_m$ has a Borel colored dimension $d(24 b_{n+1}^2)$ witness
  at scale $24 b_{n+1}^2$ with diameter $r(24 b_{n+1}^2)$,
  so we are done since $d(24 b_{n+1}^2) = d(24 r(24 b_n^2)^2) \le b_n$
  and $r(24 b_{n+1}^2) = b_{n+2}$.
\end{proof}

\section{Local stability}\label{sec:local_stability}

The goal of this section is to get some ``local,''
quantitative analogues of the functions $t$ and $T$ from \zcref{defn:go_to_stable}. In particular we will need a version where $t$ is never too big.
We specialize to the monoids $\N^d$ for $d \in \N$. The main reason is that we will depend heavily on the following variant of Steinitz's Lemma. 
Here and for the rest of the paper, norms and distances in $\R^d$ are $\ell_\infty$. 

The reader should be warned that we will use additive notation for the monoids $\N^d$. Unfortunately this introduces some inconsistency with the rest of the paper and, as usual, makes the notation for actions a bit confusing. 

\begin{lem}\label{lem:dog_walker_convex}
    Let $d \in \N$ and $x,y \in \R^d$. Let $n \in \N$ and $v_i \in \R^d$ for $i < n$ with $\|v_i\| \le 1$ and $x + \sum_{i} v_i = y$. 
    Then there is some permutation $\pi \in S_n$ such that for each $j < n$, 
    $x + \sum_{i < j} v_{\pi(i)}$ has distance at most $d$ from the line segment $\overline{x y}$. 
\end{lem}

\begin{proof}
    The situation is clearly translation invariant, so let us assume $x = 0$.
    The usual statement of Steinitz's Lemma gives the case where $y = 0$. 
    See e.g. \cite[Miniature 20]{matouvsek2010thirty}. 
    If $y \neq 0$, let $m = \lceil \|y\| \rceil$ and pick $\alpha_i \in [0,1/\|y\|]$ for $i < m$ with $\sum_i \alpha_i = 1$. Then set $w_i =-\alpha_i y$ so that $\|w_i\| \le 1$ and the $v_i$'s and $w_i$'s sum to 0. 

    Therefore there is some ordering of the $v_i$'s and $w_i$'s for which the sum of any initial segment has norm at most $d$. Now consider the induced ordering of the $v_i$'s. The sum of any initial segment of this ordering, call it $a$, differs from one of the above sums, call it $b$, by a sum of $w_i$'s, i.e., a vector of the form $\beta y$ for $\beta \in [0,1]$. 
    That is, $a = b + \beta y$, so since $\|b\| \le d$ and $\beta y \in \overline{0y}$
    this is as desired. 
\end{proof}

\begin{defn}
    Let $d \in \N$ and $\Gamma \le \Z^d$. 
    Let $\sim_\Gamma$ denote the congruence on $\N^d$ given by $g \sim_\Gamma h \iff g - h \in \Gamma$. 
    Abusing notation slightly, let $\N^d/\Gamma := \N^d /\sim_{\Gamma}$.
\end{defn}

The justification for this abuse of notation is that, by definition, $\N^d/\Gamma$ is exactly the (cancellative commutative) monoid given by the image of $\N^d$ in the quotient group $\Z^d/\Gamma$.

\begin{defn}\label{defn:r-stable}
    Let $r,d \in \N$, $X$ an $\N^d$-set, and $x \in X$. 
    \begin{itemize}
        \item $\sim_x^r := \sim_x \upharpoonright [0,r]^d$. 
        \item $\Gamma_x^r := \langle g - h \mid g \sim_x^r h \rangle \le \Z^d$. 
        \item $x$ is called \emphd{$r$-stable} if 
        \begin{enumerate}
            \item $\sim_{\Gamma^r_x} \subseteq \sim_x$
            \item $\sim^r_x = \sim^r_{kx}$ for all $k \in [0,r]^d$.
        \end{enumerate}
        Let $\cS_r(X) := \{x \in X \mid x \text{ is } r \text{-stable}\}$.
    \end{itemize}
\end{defn}

By \zcref{lem:sim_mono}, condition (2) is equivalent to $\sim^r_x \supseteq \sim^r_{r \1 x}$, where $\1 \in \N^d$ denotes the all 1's vector.

For $\sim$ a congruence on $\N^d$, $X$ an $(\N^d/\sim)$-set, and $r \in \N$, let us write \emphd{$r$-free} in place of $[0,r]^d$-free (\zcref{defn:R-free}). Similarly we write $\cF_r$ in place of $\cF_{[0,r]^d}$. 
The definition of $r$-stable is engineered to make the following true.

\begin{lem}\label{lem:stable_descend}
    Let $r,d \in \N$, $X$ an $\N^d$-set, and $x \in \cS_r(X)$.
    The action $\N^d \car \N^d x$ descends to an action of $\N^d/\Gamma^r_x$, and $x$ is an $r$-free point with respect to this action.
\end{lem}

\begin{proof}
    The first statement is immediate from $\sim_{\Gamma^r_x} \subseteq \sim_x$ and \zcref{lem:action_descend}. 
    For the second, let $g,h,k \in [0,r]^d$ such that $(g+k) x = (h + k) x$.
    We need to show $g - h \in \Gamma^r_x$.
    We have $g \sim_{kx} h$ by definition, and so $g \sim_x h$ by (2) in the definition of $r$-stable, and so $g-h \in \Gamma^r_x$ by definition of $\Gamma^r_x$. 
\end{proof}

\begin{lem}\label{lem:stable_convex}
    Let $r,d \in \N$, $X$ an $\N^d$-set, and $x \le w \le y \in X$. If $x,y \in \cS_r(X)$ and $\Gamma^r_x = \Gamma^r_y$, then $w \in \cS_r(X)$ and $\sim^r_x = \sim^r_w = \sim^r_y$.
\end{lem}

\begin{proof}
    First, $\sim^r_x = \sim^r_y$ since both consist of the pairs $g,h \in [0,r]^d$ with $g - h \in \Gamma^r_x$. (One direction here is by (1) in the definition of $r$-stable.)
    By \zcref{lem:sim_mono} $\sim^r_x \subseteq \sim^r_w\subseteq \sim^r_y$, so these are all equalities. 

    It remains to show that $w$ is $r$-stable. The equalities above show $\Gamma^r_w = \Gamma^r_x$, so $\sim_{\Gamma^r_w} \subseteq \sim_x \subseteq \sim_w$, giving (1). 
    For (2) we have $\sim^r_{r\1 w} \subseteq \sim^r_{r\1 y} = \sim^r_y = \sim^r_w$. 
\end{proof}

The following condition will be used to find stable points. 

\begin{lem}\label{lem:stable_of_same_congruence}
    Let $r,d \in \N$, $X$ an $\N^d$-set, and $x \in X$.
    If $\sim_x^r \supseteq \sim_{((d+2)r\1)x}^r$, then $((d+1)r\1)x$ is $r$-stable. 
\end{lem}

\begin{proof}
    By \zcref{lem:sim_mono}, $\sim_x^r \subseteq \sim_{((d+1)r\1)x}^r \subseteq \sim^r_{((d+2)r \1)x}  \subseteq \sim_x^r$, 
    so these are equal and we have (2) in the definition of stable. This equality also implies $\Gamma^r_{((d+1)r\1)x} = \Gamma^r_x$.

    For (1), let $g,h \in \N^d$ with $g - h \in \Gamma^r_x$. We need to show $g + (d+1)r\1 \sim_x h + (d+1)r\1$. 
    Let $n \in \N$ and $g_i \sim_x^r h_i$ for $i < n$ such that $g - h = \sum_i (g_i - h_i)$. 
    For $j \le n$, let
    $\delta_j = \sum_{i < j} g_i - h_i$.
    For each $i$, $g_i - h_i \in [-r,r]^d$.
    Therefore
    by \zcref{lem:dog_walker_convex}, we may arrange so that for each $j$, $h + \delta_j$ is within $dr$ of the positive quadrant $[0,\infty)^d \subseteq \R^d$. Equivalently,
    $h + \delta_j + dr\1 \in \N^d$. 

    Now fix $j < n$.
    Since $h_j \in [0,r]^d$ we have
    $h + \delta_j + (d+1)r\1 - h_j \in \N^d$.
    Adding this to both sides of the congruence $g_j \sim_x h_j$ gives
    \[h + \delta_{j+1} + (d+1)r\1 = h + \delta_j + (d+1)r\1 + g_j - h_j \sim_x h + \delta_j + (d+1)r\1.\]
    Therefore by transitivity we have
    \[ h + (d+1)r\1 = h + \delta_0 + (d+1)r\1 \sim_x h + \delta_n + (d+1)r\1 = g + (d+1)r\1,\]
    as desired.
    
\end{proof}

\begin{cor}\label{cor:dog}
    Let $r,d \in \N$, $X$ an $\N^d$-set, and $x \in X$. There is $t \in \N$ with $t \le \poly_d(r)$ such that $t\1  \cdot x$ is $r$-stable. 
\end{cor}

\begin{proof}
    For $i \in \N$, let $E_i = \sim_{i(d+2)r\1  \cdot x}^r$. 
    For each $i$,
    $E_i \subseteq E_{i+1} \subseteq ([0,r]^d)^2$, so there is some $i \le (r+1)^{2d}$ for which $E_i = E_{i+1}$.
    Then we can take $t = (i(d+2) + d+1)r$
    by \zcref{lem:stable_of_same_congruence}. 
\end{proof}

This lets us define the following functions. The first two are analogous to those from \zcref{defn:go_to_stable}. 

\begin{defn}
    Let $r,d \in \N$, $X$ an $\N^d$-set, and $x \in X$. 
    \begin{itemize}
        \item $t^r(x)$ denotes the minimal $t \in \N$ with the property from \zcref{cor:dog} ($t \1  \cdot x$ is $r$-stable). 
        \item $T^r(x) := t^r(x) \1  \cdot x \in \cS_r(X)$. 
        \item $\Delta^r_x := \Gamma^r_{T^r(x)}$. 
    \end{itemize}
\end{defn}

\begin{lem}\label{lem:stable_lipschitz}
    Let $s,r,d \in \N$ with $s \le r$, $X$ an $\N^d$-set, and $x,y,z \in X$ such that $x,y \in [0,s]^d \cdot z$.
    \begin{enumerate}
        \item If $t^r(x) \ge t^r(y)$, then $\Delta^r_x \ge \Delta^r_y$.
        \item If $t^r(x) > t^r(y) + s$, then $\Delta^r_x \gneq \Delta^r_y$. 
    \end{enumerate}
\end{lem}

\begin{proof}
    (1): We have $r\1  \cdot x \ge y$, so $r\1  \cdot T^r(x) = (r + t^r(x))\1  \cdot x \ge t^r(x) \1  \cdot y \ge t^r(y) \1  \cdot y = T^r(y)$. Thus $\sim^r_{T^r(x)} = \sim_{r\1  \cdot T^r(x)}^r \supseteq \sim^r_{T^r(y)}$, the equality since $T^r(x)$ is $r$-stable. Since $\Delta^r_x$ is generated by the differences of pairs in $\sim^r_{T^r(x)}$ and likewise for $y$, we are done. 

    (2): Suppose not. Then by (1), $\Delta^r_x = \Delta^r_y$.
    Let $w = (t^r(y) + s)\1  \cdot x$. 
    We will show $w$ is $r$-stable, contradicting the definition of $t^r(x)$. 
    Since $s\1  \cdot x \ge y$, we have 
    \[T^r(y) = t^r(y) \1  \cdot y \le (t^r(y) + s)\1  \cdot x = w \le T^r(x).\] 
    Therefore $w$ is $r$-stable by \zcref{lem:stable_convex}.
\end{proof}

At this point the reader can forget about the definition of $r$-stable. The important things are the consequences in
\zcref{cor:dog,lem:stable_descend, lem:stable_lipschitz}.

From \zcref{lem:stable_lipschitz}(1), we conclude that the image of $[0,r]^d \cdot z$ under $\Delta^r$ is a chain of subgroups of $\Z^d$. Furthermore, each of these subgroups is generated by elements of $[-r,r]^d$. The following then establishes that not too many subgroups can appear in this image. 

\begin{lem}\label{lem:chain_of_subgroups}
    Let $r,d \in \N$.
    Let $\Delta_1 \lneq \Delta_2 \lneq \cdots \lneq \Delta_l \le \Z^d$ with each $\Delta_i$ generated by elements of $[-r,r]^d$. Then $l = O_d(\log(r))$. 
\end{lem}

\begin{proof}
    First we may assume that each $\Delta_i$ has the same rank, say $\rho$, since this only adds a factor of $d$. Then $\Span_\R(\Delta_i)$ is constant, call it $V$. Let $L = V \cap \Z^d$. By the fourth isomorphism theorem, the quotients $\Delta_i /\Delta_1$ form a strictly increasing chain of subgroups of $L/\Delta_1$, so it suffices to show that $|L/\Delta_1| \le \poly_d(r)$. 

    Let $g_1,\ldots,g_\rho \in \Delta_1 \cap [-r,r]^d$ be an $\R$-basis for $V$. Let $P \subseteq V$ be the parallelepiped given by $g_1,\ldots,g_\rho$. Then for any $x \in V$, there is some $y \in P$ such that $x + \Delta_1 = y + \Delta_1$. Since $\Delta_1 \le \Z^d$, if $x \in L$ then so must be $y$. Therefore $|L/\Delta_1| \le |P \cap L| = |P \cap \Z^d|$. Since $P \subseteq [-\rho r,\rho r]^d$ and $\rho \le d$, this last number is at most $(2dr+1)^d$. 
\end{proof}

\section{Proof of \zcref{thm:main}}
\label{sec:proof-main}

The main goal is now to prove the following. The monoids below will always be quotients of $\N^d$, and we will always use the generating set $[0,1]^d$. 
(Or rather, the image of $[0,1]^d$ in the quotient, but we will not be explicit about this distinction below.)
Accordingly, we will write $G^X$ in place of $G_{[0,1]^d}^X$ for $X$ an $\N^d$-set. 

\begin{lem}
    Let $k,d \in \N$. Let $X$ be a Borel $\N^d$-set for which the action of each $g \in [0,1]^d$ is at most $k$-to-one. 
    Then $G^X$ has Borel uncolored $(O_{d,k}(\log (s)), \poly_{d,k}(s))$-dimension growth. 
\end{lem}

\begin{proof}
Let $s \in \N$ be a multiple of 16. 
By \zcref{cor:monoid_asdim_local}, there are $R,C \in \poly_d(s)$ and $D \in O_d(1)$ such that 
\[\Pi_{\dim_{[0,s]^d, [0,R]^d, D}} \in \LOCAL_{[0,1]^d}(C \log^*(n))\]
with respect to any monoid of the form $\N^d / \Gamma$, $\Gamma \le \Z^d$. 

Then, by \zcref{prop:local_to_borel}, there is $r \in \poly_{k,d}(s)$
such that for any Borel action 
$\N^d/\Gamma \car Y$ where each $g \in [0,1]^d$ is at most $k$-to-one, 
there is a Borel $l : Y \to [0,R]^d$ which is a $\Pi_{\dim_{[0,s]^d, [0,R]^d, D}}$-labeling of $\cF_{r}(Y)$. 
Then by \zcref{lem:dimension_lcl_to_witness_point},
the Borel function $f : Y \to Y$ given by $f(y) = (l(\frac{s}{2}\1 y) + \frac{s}{2}\1)y$ 
sends the $\frac{s}{2}$-ball in $G^Y$ centered at any point in $\cF_{r}(Y)$ to at most $D + 1$ points. 
Note that for all $y \in Y$, $f(y) \in [0,R+s]^d \cdot y$.

We will now apply the constructions from \zcref{sec:local_stability}. Let $T = T^{r} : X \to \cS_{r}(X)$ and similarly for $t,\Gamma, \Delta$. 
For $\Gamma \le \Z^d$ generated by elements of $[-r,r]^d$, define 
\begin{itemize}
    \item $\cS_\Gamma := \{x \in \cS_{r}(X) \mid \Gamma_x = \Gamma\}$.
    \item $\cT_\Gamma := \{x \in X \mid \Delta_x = \Gamma\} = T\inv(\cS_\Gamma)$.
    \item $Y_\Gamma := \N^d \cS_\Gamma$. 
\end{itemize}
By \zcref{lem:stable_descend}, the action $\N^d \car Y_\Gamma$ descends to one of $\N^d/\Gamma$, and every point in $\cS_\Gamma$ is $r$-free with respect to this action. This action is still Borel and at most $k$-to-one for elements of $[0,1]^d$, so
we may fix for each $\Gamma$ 
a Borel $f_\Gamma : Y_\Gamma \to Y_\Gamma$ which sends the $\frac{s}{2}$-ball in $G^{Y_\Gamma}$ centered at any point of $\cS_\Gamma$ to at most $D+1$ points and which sends each point $x$ to some point in $[0,R+s]^d \cdot x$. 

Noting that the $\cS_\Gamma$'s are pairwise disjoint, let $f = \bigcup_\Gamma (f_\Gamma \uhr \cS_\Gamma)$. 
We claim $F := f \circ T \circ \frac{s}{16}\1$ is an asymptotic dimension $O_{d,k}(\log(s))$ witness at scale $\frac{s}{16}$ with radius $\poly_{d,k}(s)$. 
Note that $\text{ran}(T) \subseteq \dom(f)$ and that $F$ is Borel.

By \zcref{cor:dog}, $T(x) \in ([0,\poly_{d,k}(s)]\1) \cdot x$, and by construction $f$ moves each point in its domain by at most $R+s$, so since $R \in \poly_d(s)$ the claim about the radius is clear.

Now fix $z_0 \in X$ and let us show $|F[B_{G^X}(z_0,\frac{s}{16})]| \le O_{d,k}(\log(s))$.
As in the proof of \zcref{lem:dimension_lcl_to_witness_point},
$\frac{s}{16}\1 \cdot B_{G^X}(z_0, \frac{s}{16}) \subseteq [0, \frac{s}{8}]^d \cdot z_0$, 
so it suffices to show the image of $[0, \frac{s}{8}]^d z_0$ under $f \circ T$ has size at most $O_{d,k}(\log(s))$. 
We may assume $s \le r$, and so by \zcref{lem:chain_of_subgroups} 
and the paragraph before it
there are at most 
$O_d(\log(r)) = O_d(\log(\poly_{k,d}(s))) = O_{d,k}(\log(s))$
many $\Gamma$ for which
$[0, \frac{s}{8}]^d \cdot z_0$ meets $\cT_\Gamma$. 
Fix such a $\Gamma$ and let $Z = \cT_\Gamma \cap [0,\frac{s}{8}]^d \cdot z_0$.
It now suffices to show that $|(f \circ T)[Z]| \le O_d(1)$. 
We will show it has size at most $D+1$. 

Let $m := \max_{x \in Z} t(x)$. 
By \zcref{lem:stable_lipschitz}, for all $x \in Z$, $t(x) \in [m - \frac{s}{8}, m]$.
Now let $y_0 = ((m + \frac{s}{8}) \1) z_0$. 
For each $x \in Z$, $(\frac{s}{8} \1)z_0 \in [0,\frac{s}{8}]^d x$, so $y_0 \in [0,\frac{s}{8}]^d \cdot (m\1 x) \subseteq [0,\frac{s}{4}]^d \cdot T(x)$.
Thus, $T[Z]$ is contained in the $\frac{s}{4}$-ball around $y_0$ in $G^{Y_\Gamma}$.
Fix an arbitrary $x_0 \in Z$. The above shows that
$T[Z]$ is contained in the $\frac{s}{2}$-ball around $T(x_0)$.
Since $T[Z] \subseteq \cS_\Gamma$, its image under $f$ is its image under $f_\Gamma$, which then has size at most $D+1$ by our choice of $f_\Gamma$. 
    
\end{proof}

\begin{proof}[Proof of \zcref{thm:main}]
  Let $M$ be a countable commutative monoid and $X$ a bounded-to-one Borel $M$-space. 
  We may assume $M = \N^{\oplus \omega}$ since it is a quotient of this monoid.
  Let $G_n = G^X_{[0,1]^n}$, so that $G_n$ is the Schreier graph of a bounded-to-one Borel $\N^n$-set. 
  By the lemma,
  each $G_n$ has Borel $(O(\log(s)), \poly(s))$-dimension growth, 
  so we are done by \zcref{increasing-union-simple}. 
\end{proof}

\section*{Acknowledgments}

This project grew out of extensive conversations between the authors and Petr Naryshkin, who proved \zcref{thm:free-fg-asdim-intro} independently. 
We wish to acknowledge his contributions to the project and express our gratitude. 

Thanks to Andrew Marks for suggesting the inclusion
of the material on chromatic numbers in \zcref{sec:ssi}.
Thanks to Clinton Conley, Su Gao, and Zolt\'{a}n Vidny\'{a}nszky for helpful conversations.

FW is supported by the National Science Foundation under grant DMS-2402064. 
JY is supported by the National Natural Science Foundation of China grants 12371343 and 12525110 (PI: Hehui Wu).

\printbibliography
\end{document}